\documentclass[12pt]{amsart}

\usepackage[english]{babel}

\usepackage[symbol]{footmisc}

\usepackage[margin=1in,footskip=.2in]{geometry}\usepackage{amsmath,amssymb,amscd,amsfonts,mathtools,tikz,caption,float,tikz-cd, enumitem}
\usetikzlibrary{patterns}
\usepackage[colorlinks = true,
            linkcolor = blue,
            urlcolor  = blue,
            citecolor = blue,
            anchorcolor = blue]{hyperref}
\usepackage[capitalise]{cleveref}
\crefname{equation}{}{}
\usepackage{mathptmx}
\usepackage{mathrsfs}
\usepackage{comment}

\newtheorem{thm}{Theorem}[section]
\newtheorem{lem}[thm]{Lemma}
\newtheorem{defn}[thm]{Definition}
\newtheorem{prop}[thm]{Proposition}
\newtheorem{coro}[thm]{Corollary}
\newtheorem{rmk}[thm]{Remark}

\newtheorem{notation}[thm]{Notation}

\numberwithin{equation}{section}

\newcommand{\norm}[1]{\left\Vert#1\right\Vert}
\newcommand{\abs}[1]{\left\vert#1\right\vert}

\newcommand{\eps}{\varepsilon}

\newcommand{\R}{\mathbb{R}}

\newcommand{\Q}{\mathbb{Q}}
\newcommand{\Z}{\mathbb{Z}}
\newcommand{\N}{\mathbb{N}}

\newcommand{\PP}{{\mathbb{P}}}

\newcommand{\origin}{O}
\newcommand{\sgcd}{S\text{-}\mathrm{gcd}}
\newcommand{\one}{\mathbf{1}}

\newcommand{\SL}{\mathrm{SL}}
\newcommand{\GL}{\mathrm{GL}}

\newcommand{\G}{\Gamma}

\newcommand{\UL}{\mathrm{UL}}

\newcommand{\Mat}{\mathrm{Mat}}

\newcommand{\Id}{\mathrm{Id}}

\newcommand{\SG}{\mathsf{G}}

\newcommand{\SA}{\mathsf{A}}
\newcommand{\SK}{\mathsf{K}}

\newcommand{\sg}{\mathsf{g}}
\newcommand{\sh}{\mathsf{h}}

\newcommand{\su}{\mathsf{u}}

\newcommand{\T}{{\mathsf{T}}}

\renewcommand{\t}{\mathsf{t}}

\newcommand{\sx}{\mathsf x}
\newcommand{\sa}{\mathsf a}
\renewcommand{\sb}{\mathsf b}
\renewcommand{\sc}{\mathsf c}
\newcommand{\sd}{\mathsf d}
\newcommand{\sv}{\mathsf v}

\newcommand{\sy}{\mathsf y}
\newcommand{\sz}{\mathsf z}

\newcommand{\bx}{{\bf x}}
\newcommand{\by}{{\bf y}}

\newcommand{\be}{{\bf e}}
\newcommand{\bv}{{\bf v}}
\newcommand{\bw}{{\bf w}}

\newcommand{\bp}{{\bf p}}
\newcommand{\bq}{{\bf q}}

\newcommand{\bsv}{\boldsymbol{\mathsf{v}}}
\newcommand{\bsx}{\boldsymbol{\mathsf{x}}}
\newcommand{\bsy}{\boldsymbol{\mathsf{y}}}

\newcommand{\bbm}{\begin{bmatrix}}
\newcommand{\ebm}{\end{bmatrix}}
\newcommand{\bpm}{\begin{pmatrix}}
\newcommand{\epm}{\end{pmatrix}}
\newcommand{\bsm}{\left(\begin{smallmatrix}}
\newcommand{\esm}{\end{smallmatrix}\right)}
\newcommand{\bsbm}{\left[\begin{smallmatrix}}
\newcommand{\esbm}{\end{smallmatrix}\right]}

\providecommand{\vol}{\mathrm{vol}}
\newcommand{\tp}[1]{^\mathrm{t}{#1}}
\renewcommand{\mod}{\mathrm{mod}\;}
\providecommand{\Interval}{\mathrm{I}}
\newcommand{\diag}{\mathrm{diag}}
\newcommand{\supp}{\mathrm{supp}}
\newcommand{\rd}{\mathrm{d}}
\newcommand{\LI}{\Omega}
\newcommand{\LD}{\Omega}

\begin{document}
\title{Mean value theorems for the S-arithmetic primitive Siegel transforms}
\author{Samantha Fairchild}
\address{Eindhoven University of Technology, Department of Mathematics and Computer Science}
\email{\href{matilto:s.k.fairchild@tue.nl}{s.k.fairchild@tue.nl}}
\author{Jiyoung Han\textsuperscript{*}}
\address{Department of Mathematics Education, Pusan National University, Busan 46241, Republic of Korea}
\email{jiyoung.han@pusan.ac.kr}
\thanks{\textsuperscript{*}Corresponding author.}

\begin{abstract}
We develop the theory and properties of primitive unimodular $S$-arithmetic lattices in $\Q_S^d$ by giving integral formulas in the spirit of Siegel’s primitive mean value formula and Rogers' and Schmidt's second moment formulas. When $d=2$, unlike in the real case, functions arising from the $S$-primitive Siegel transform are unbounded, requiring a careful analysis to establish their integrability.

We then use mean value and second moment formulas in three applications. First, we obtain quantitative estimates for counting primitive $S$-arithmetic lattice points. We next establish a quantitative Khintchine--Groshev theorem, which, in the real case, involves counting primitive integer points in $\mathbb{Z}^d$ subject to congruence conditions. 
Finally, we derive an $S$-arithmetic logarithm law for unipotent flows in the spirit of Athreya--Margulis.
These applications follow the spirit of the real case, but require new technical aspects of the proofs, particularly when $d=2$. 
\end{abstract}
\maketitle
\setcounter{tocdepth}{1}

\section{Introduction}\label{sec:introduction}
The main work of this paper is proving second moment formulas for the primitive Siegel integral formula in the $S$-arithmetic setting. The classical Siegel--Veech formula formalizes the idea that the expected number of lattice points in $\R^2$ in the ball of radius $R$ is $\pi R^2$. Namely for $d\geq 2$ the space of unimodular lattices $g\Z^d$ is parametrized by $g\SL_d(\Z)\in\SL_d(\R)/\SL_d(\Z)$, which inherits a Haar probability measure. Given $f:\R^d\to \R$ with compact support, define the \textbf{Siegel transform}
$$\widetilde{f}(g) = \sum_{\bv \in \Z^d-\{0\}} f(g\bv),$$
which counts the number of lattice points in $B(0,R) = \{\bx\in \R^d: \norm{\bx}_2\leq R\}$ when $f= \mathbf{1}_{B(0,R)}.$
The \textbf{Siegel integral formula} \cite{Siegel45} 
$$\int_{\SL_d(\R)/\SL_d(\Z)} \widetilde{f}(g) \,d g = \int_{\R^d} f(\bx)\,d\bx$$
gives the expected value in terms of the Lebesgue volume in $\R^d$.

A natural question from the above expected value formula is asking about higher moments. In this case for $d\geq 3$, $\widetilde{f} \in L^k$ for all  $1\leq k\leq d-1$ provided that $f$ is semicontinuous and of compact support, for instance (see Remark \ref{rmk: range of ftns}). \cite{Rogers55, Schmidt60}, and Rogers gave explicit formulas. These formulas and their applications have been generalized to many different settings including $S$-arithmetic numbers \cite{HLM2017, Han21}, Adelic numbers \cite{Kim22Adelic}, rational points on Grassmanians \cite{Kim22Grass}, as well as affine and congruence lattices \cite{AGH21, EMV15, GKY22}.

The case of $d=2$ is of particular interest, since $\widetilde{f}\:^k$ is not integrable for any $k \ge 2$. However when we consider primitive vectors $P(\Z^d) = \{\bv\in \Z^d: \gcd(\bv)=1\}$, the corresponding 
\textbf{primitive Siegel transform} given by
 \begin{equation}\label{eq:psvt}\widehat{f}(g) = \sum_{\bv\in P(\Z^d)} f(g\bv)\end{equation}
 satisfies $\widehat{f} \in L^k$ for all $k\in \N$  when $d=2$. Moreover we have the \textbf{primitive Siegel integral formula} \cite{Siegel45} for $f: \R^d\to \R$ 
 \begin{equation}\label{eq:primsiegel}\int_{\SL_d(\R)/\SL_d(\Z)} \widehat{f}(g) \,d g = \frac{1}{\zeta(d)}\int_{\R^d} f(\bx)\,d\bx.\end{equation}
 Here $\zeta$ is the Riemann zeta function, which arises since the set of primitive integers has density $\zeta(d)^{-1}$ in $\Z^d$.
The analogous work of Rogers for higher moments in the primitive case was completed by Schmidt \cite{Schmidt60} in the case of $d=2$. The story when $d=2$ has many generalizations and applications most notably in the case of translation surfaces due to the seminal work of Veech providing the analogous statement of \cref{eq:primsiegel} in \cite{Veech98}. More recent work in translation surfaces also includes higher moments with applications from the second moment arising in \cite{ACM19, AFM22, Fai21, BF22Pairs}.

\subsection*{Integral Formulas} Inspired by the works above, we consider the primitive integral formulas in the $S$-arithmetic setting for $d\geq 2$ and for $S$, {where $S$ is the union} of $\{\infty\}$ and finitely many distinct primes $\{p_1,\ldots, p_s\}$. 
The $S$-arithmetic setting is interesting in its own right, as we combine both the Archimedean and finitely many distinct non-Archimedean places when considering possible closures of $\Q$. A key distrinction arises when $d=2$: unlike \eqref{eq:psvt} where $\hat{f}$ is bounded and hence in $L^k(\SL_2(\R)/\SL_2(\Z))$ for any $k$, functions arising from the $S$-primitive Siegel transform are unbounded, thus proving integrability requires careful analysis.

More generally, the $S$-arithmetic subgroups arise naturally when considering finitely generated subgroups of $\GL_d(\overline{\Q})$. The first main result gives a primitive mean value formula over the $S$-arithmetic numbers $\Q_S$, where the 
$S$-arithmetic lattices we consider are parameterized by $\SG_d/\Gamma_d$, where $\SG_d=\SL_d(\Q_S)$ and $\Gamma_d=\SL_d(\Z_S)$. We now state the result for the primitive Siegel transform $\widehat{f}$, similar to the definition in \cref{eq:psvt} where we instead sum over primitive $S$-arithmetic vectors $P(\Z_S^d)$ defined in \cref{sec:primvectors}. Analogous to the real case, the density of the set of primitive integers comes into play with the $S$-arithmetic $\zeta$-function: 
\begin{equation}\label{eq:zetaS}
    \zeta_S(d)=\sum_{\{m\in \N:\gcd(m,p_1\cdots p_s)=1\}}\; \frac{1}{m^d}.
\end{equation}
The statements also use the Haar probability measure $\mu_d$ on $\SG_d/\Gamma_d=\SL_d(\Q_S)/\SL_d(\Z_S)$, and the volume measure $d\bsx$ on $\Q_S^d$ that we define in \cref{sec:notation}. 

\begin{prop}[$S$-arithmetic primitive mean value formula]\label{prop:primmean}
	Let $f \in B_c^{SC}(\Q_S^d) $. For any $d \geq 2$, {the $S$-primitive Siegel transform $\widehat{f}$ defined as in \eqref{def: S-primitive Siegel transform}} is integrable with 
	$$\int_{\SG_d/\Gamma_d} \widehat{f}(\sg\Gamma_d) \,d\mu_d(\sg) = \frac{1}{\zeta_S(d)}\int_{\Q_S^d} f(\bsx)\,d\bsx.$$
\end{prop}
Here $f\in B_c^{SC}(\Q_S^d)$ denotes a bounded semicontinuous function with compact support, see \cref{sec:meanvalues} and \cite[\S~6]{BF22Pairs} for more discussion on this choice. In \cref{sec:secondmomentformulas}, we introduce the second moment primitive integral formula for $d=2$. For now, we give a second moment formula for $d\geq 3$ which mirrors the work of \cite{Han21} on a (non-primitive) $S$-arithmetic Rogers' formula.
\begin{thm}[$S$-arithmetic primitive second moment formula when $d\geq 3$]\label{thm:primrog}
 If $F\in B_c^{SC}((\Q_S^d)^2)$ for a fixed $d\geq 3$, then 
	$$\widehat{F}(\sg\Gamma_d) = \sum_{(\bv^1,\bv^2) \in P(\Z_S^d)\times P(\Z_S^d)} F(\sg\bv^1, \sg \bv^2)$$
	satisfies $\widehat{F} \in L^1(\SG_d/\Gamma_d)$ and
	$$\int_{\SG_d/\Gamma_d} \widehat{F}(\sg\Gamma_d)\,d\mu_d(\sg) = \frac{1}{\zeta_S(d)^2} \iint_{\Q_S^d\times \Q_S^d} F(\bsx, \bsy) \,d\bsx \,d\bsy + \frac{1}{\zeta_S(d)} \sum_{k\in \Z_S^{\times}} \int_{\Q_S^2} F(\bsx, k\bsx)\,d\bsx.$$
Here $\Z_S^\times$ is the set of all units of $\Z_S$, which reduces to $\Z^\times = \{\pm 1\}$ when $S=\{\infty\}$.
\end{thm}

The case when $d=2$ requires more care, and is stated in two forms after more notation is established. The first form is contained in \cref{rank 2 1st form} following the strategy of \cite{Fai21}, which uses a folding-unfolding argument to decompose $\Q_S^2 \times \Q_S^2$ into $\Gamma_2$-orbits. We build from the first form to obtain the second form in \cref{integral over the cone} following work of \cite{BF22Pairs,Schmidt60} by integrating over a cone which allows for a main term of the the integral that is almost as simple as the main term in \cref{thm:primrog}. Along the way we highlight \cref{Euler ftn-sum formula} which is of independent interest as we give asymptotic expansions of the Euler summatory function over integers with an added congruence condition.
 
\subsection*{Applications} 
We highlight three applications of the primitive integral formulas in the $S$-arithmetic setting. The first two introduce a flavor of the results by giving Schmidt's counting theorem and a quantitative Khintchine--Groshev theorem for real lattices with both primitive and congruence conditions. Further applications which adhere to the $S$-arithmetic context will be stated in \cref{sec:errorterms} and \cref{sec:KG}. The third application gives logarithm laws for unipotent one-parameter subgroups in the $S$-arithmetic setting. 

\subsubsection*{Counting}The first application uses \cref{prop:primmean} and the second moment formulas to obtain asymptotic estimates on counting lattice points. Our main result is \cref{Schmidt Main Theorem}. The proof follows the general outline of \cite{Schmidt60}, with new ideas coming from finding the correct extension to the $p$-adic places. We state here an application of \cref{Schmidt Main Theorem} to the real case. 

\begin{thm}\label{Schmidt Prim+Cong Theorem}
Let $d\ge 3$. Fix an increasing family of Borel sets $\{A_{T}\}_{T\in \R_{>0}} \subset \R^d$ with $\vol(A_{T})=T$.
Let $N=p_1^{k_1}\cdots p_s^{k_s}\in \N$ for finitely many distinct primes $p_1, \ldots, p_s$ and $k_i\in \N$. Fix $\bv_0\in P(\Z^d)$. Set
\[
P_{\bv_0, N}(\Z^d):=\{\bv\in P(\Z^d): \bv\equiv \bv_0\;\mod \;N\}.
\]
For any $\delta\in (\frac 2 3, 1)$, it follows that for almost all $g \in \SL_d(\R)$, 
\[
\# \left( g P_{\bv_0, N}(\Z^d)\cap A_{T}\right)
=\frac {T} {N^d \zeta_S(d)} + O_{g}\left(T^\delta\right).
\]
\end{thm}
The reduction from theorems over an $S$-arithmetic space to results over the real field, particularly in the context of counting integer vectors with congruence conditions, does not seem to have been previously recorded in literature. It was suggested by an anonymous referee that \cref{Schmidt Main Theorem} may have applications in the real setting, though without reference to congruence conditions. Motivated by this comment and drawing on earlier conversations with Seungki Kim and Anish Ghosh, we recognized that \cref{Schmidt Main Theorem} can indeed be applied to obtain \cref{Schmidt Prim+Cong Theorem}.

It remains open to develop an analogous statement of \cref{Schmidt Prim+Cong Theorem} for $d=2$. This comes from the fact that the error term for $d=2$ in \cref{Schmidt Main Theorem} has an interesting form different from typical counting results as we keep track of two exponents $\delta_1$ and $\delta_2$. This requires in particular that the $p$-adic part must also have increasing volume in the construction of an increasing family of sets. 

\subsubsection*{Diophantine Approximation}The second application is related to Diophantine approximation. Given a function $\psi: \R_{\geq 0}\to \R_{\geq 0}$, we say that an $m\times n$ matrix $A$ is \emph{$\psi$-approximable} if there are infinitely many nonzero $(\bp , \bq) \in \Z^m\times \Z^n$ so that 
$$\|A\bq - \bp\|^m \leq \psi(\|\bq\|^n).$$
The classical Khintchine--Groshev theorem gives a criterion on $\psi$ for understanding the density of $\psi$-approximable numbers. 
The problem quantifying the theorem in the divergent case has studied in various settings and various methods (see \cite{Schmidt60a, Harman98, Harman2003, KS21, AS24} for instance).

We point out that \cite{AGY21} established a quantitative Khintchine--Groshev theorem with congruence conditions based on Schmidt's counting result with an error term (\cite[Theorem 1 (3)]{Schmidt60a}), and one can use \cite[Theorem 1 (4) and Theorem 2 (6)]{Schmidt60a} for counting primitive integer vectors to obtain a primitive quantitative Khintchine--Groshev theorem. However, to the best of our knowledge, obtaining the quantitative Khintchine--Groshev theorem by combining these two conditions is challenging without delving into the geometry of $S$-arithmetic numbers, as outlined below.

\begin{thm}\label{quan K--G thm with prim+cong}
Let $d=m+n\ge 3$ and fix $N=p_1^{k_1}\cdots p_s^{k_s}\in \N$ for $p_1, \ldots, p_k$ mutually distinct primes and $k_i\in \N$. Fix $\bv_0\in P(\Z^d)$. Let $\psi:\R_{>0}\rightarrow \R_{\ge 0}$ be a non-increasing function for which $\sum_{1\le q \le T} \psi(q)$ diverges.
Then for almost all $X\in \Mat_{m,n}(\R)$, 
\[
\lim_{T\rightarrow \infty}
\frac {\#\left\{(\bp, \bq)\in P(\Z^m\times \Z^n): \begin{array}{c}\|X\bq -\bp\|^m\le \psi(\|\bq\|^n),\;\|\bq\|^n<T,\\
\text{and}\;(\bp, \bq)\equiv \bv_0\;\mod\; N \end{array}\right\}}{\left(\zeta_S(d)N^{d}\right)^{-1}\sum_{1\le q\le T} \psi(q)}=1.
\]
\end{thm}

Our main contribution gives an asymptotic density for the number of $\psi$-approximable $S$-arithmetic numbers in the case when $m=n=1$. As in the case when $d=2$, we have two different error terms, so we give the exact statement in \cref{1-d primitive Khintchine-Groshev Thm} after more notation is established.

 The proof uses \cref{primitive Khintchine-Groshev Thm} which gives a condition to find the density of $\psi$-approximable $S$-arithmetic integers by using the second moment for $d\geq 3$ from \cref{thm:primrog}. \cref{primitive Khintchine-Groshev Thm} adapts the results of \cite{Han22} to the primitive setting, where \cite{Han22} in turn generalizes the method of \cite{AGY21} in the $S$-arithmetic setting. We remark that Kelmer and Yu in \cite{KY23} showed a quantitative Khintchine--Groshev theorem where the error bound refines the work of \cite{AGY21} in a more general setting, but we did not see any direct benefits of using this version instead of that in \cite{AGY21}.

\subsubsection*{Unipotent Logarithm Laws} The third application gives a theorem for logarithm laws. In the classical setting, logarithm laws give the rate of escape from a compact set for a one-parameter geodesic flow. This has been well studied in the $S$-arithmetic setting in \cite{ultrametric1, ultrametric2}. Here we consider logarithm laws for unipotent flows in the spirit of \cite{AM09}: for $\mu$-almsost every $g\in \SL_d(\R)/\SL_d(\Z)$, 
	$$\limsup_{t\to\infty} \frac{\log\alpha_1(u_tg\Z^d)}{\log t} = \frac{1}{d},$$
where $u_t$ is a unipotent flow for $\SL_d(\R)/\SL_d(\Z)$ and $\alpha_1(g\Z^d) = \sup\{{\norm{\bv}}^{-1}:0\neq \bv\in g\Z^d\}$ measures the rate of escape by the shortest vector.

We find the rate of escape in the $S$-arthmetic setting as follows. First we recall the definition of the shortest $S$-arithmetic lattice vector.
\begin{defn}
	We define $\alpha_1: \SG_d/\Gamma_d \to\mathbb{R}$ by
	$$\alpha_1(\Lambda) :=  \sup\left\{\prod_{p\in S} \norm{\bv_p}_p^{-1}:\bsv \in \Lambda-\{0\} \right\}
	=\sup\left\{\prod_{p\in S} \norm{\bv_p}_p^{-1}:\bsv \in P(\Lambda) \right\}.$$
\end{defn}
Next, we clarify our choice of neighborhood when taking limits in $\Q_S$.
\begin{defn}
	We define the limsup of a function $f:\mathbb{Q}_S \to \mathbb{R}$ by considering the following neigbhorhood of infinity in $\Q_S$
 $$\limsup_{|\sx|\rightarrow \infty} f(\sx) = \inf_{\stackrel{\sx=(x_p)_{p\in S} \in \Q_S}{|x_p|_p \to\infty,\forall p\in S}} \left(\sup\left\{f(\sy): \sy\in \Q_S, \; { |y_p|_p\ge |x_p|_p},\;\forall p\in S\right\}\right).$$
\end{defn}
In the above definition we can replace the infimum by a limit, which is well defined by monotonicity. Finally note we have a one-$\Q_S$-parameter unipotent subgroup generated by elements $\su_\sx$ for $\sx \in \Q_S$. {In this setting, we obtain a logarithm law.}

\begin{thm}\label{thm:loglawAM}
For $d\ge 2$, it follows that for {$\mu_d$}-almost every $\Lambda$, 
\[
\limsup_{|\sx|\rightarrow\infty}
\frac{\log(\alpha_1(\su_\sx \Lambda))}{\log\left(\prod_{p\in S} |x_p|_p\right)} = \frac{1}{d}.
\]
\end{thm}
A key difference from the work of  \cite{AM09} arises in the case of $d=2$, where we have access to a second moment formula in order to prove a random Minkowski theorem. Also, the additional $p$-adic places require finding the correct target sets for a lower bound and the correct scaling factor.

\subsection{Outline} The paper is organized as follows. In \cref{sec:notation} we set up the notation, and state all of the remaining main theorem statements. In particular, we established the notion of primitive $S$-arithmetic vectors and provided equivalent definitions to show that $S$-primitive vectors generalize primitive integer vectors in $\Z^d$. \Cref{Schmidt Main Theorem}, an analog of Schmidt's counting theorem, features a different error term due to the $S$-arithmetic setting. In dimension $d=2$, an additional condition on the increasing family of borel sets is required, 
which did not appear in the real case. This extra restriction excludes the two dimensional case of \Cref{Schmidt Prim+Cong Theorem} and thus \Cref{quan K--G thm with prim+cong}.

In \cref{Proofs of Primitive integral formulas} we provide proofs of the primitive integral formulas. We first achieve the integrability of functions arising from the $S$-primitive Siegel transform and then apply Riesz--Markov--Kakutani theorem. The main difficulty occurs when $d=2$, where, unlike in the real case, the transformed functions are unbounded. To overcome this, we employ analytic techniques to show that these functions lie in $L^1(\SL_d(\Q_S)/\SL_d(\Z_S))$ and $L^2(\SL_d(\Q_S)/\SL_d(\Z_S))$.
In \cref{Integral formulas over Cone}, we define the notion of a cone in the $S$-arithmetic space and extend these formulas in the case of $d=2$ to obtain variance estimates.

We conclude in \cref{sec:applicationproofs} with the proofs of the three applications, separated into three subsections: error terms in \cref{sec:errortermsproofs}, Khintchine--Groshev theorems in \cref{sec:KGproofs}, and logarithm laws for unipotent flows in \cref{sec:loglawsproofs}. For the last application, we classify unipotent flows in $\SL_d(\Q_S)$ up to conjugacy, by showing that any unipotent one-parameter subgroup arises from a nilpotent matrix. In the real case, this follows from the exponenetial and logarithmic maps between Lie groups and their Lie algebras. Although these maps are not globally defined in the $p$-adic Lie groups in general, we observe that they are polynomial on the sets of unipotent and nilpotent matrices, respectively. This yields the same classification of unipotent flows in $\SL_d(\Q_p)$ as in $\SL_d(\R)$.

\subsection{Acknowledgements} We would like to thank Jayadev Athreya for connecting us for this project. We also appreciate Barak Weiss and Shucheng Yu for valuable advice. SF was partially supported by the Deutsche Forschungsgemeinschaft (DFG) -- Projektnummer 445466444 and 507303619. JY was thankful for the support of Tata Institute of Fundamental Research and Korea Institute for Advanced Study. Some of the work and ideas in this project came from discussions during the conference on Combinatorics, Dynamics and Geometry on Moduli Spaces at CIRM, Luminy, September 2022.

\section{Notation and Results}\label{sec:notation}
We will focus on $S$-arithmetic groups with respect to the rational numbers. One can work with $S$-arithmetic groups in a more general setting, to which we refer the reader to a short overview with many further resources in \cite[Appendix C]{DWM}. For ease of reference, we have \cref{sec:Sarithmeticreference} and \cref{sec:unimodular} cover the background in $S$-arithmetic numbers and their unimodular lattices. In \cref{sec:primvectors} we introduce the notion of a primitive $S$-arithmetic vector, and the analog of the greatest common divisor. In \cref{sec:meanvalues} and \cref{sec:secondmomentformulas} we give the exact statements of the integral formulas. Finally we conclude the statements of the theorems for the three applications in \cref{sec:errorterms}, \cref{sec:KG}, and \cref{sec:LL}.

\subsection{S-arithmetic space} \label{sec:Sarithmeticreference} Let $S$ be a union of $\{\infty\}$ and a finite set of distinct primes $S_f = \{p_1,\ldots, p_s\}$. Let $\Q_p$ denote the completion field of $\Q$ with respect to the $p$-adic norm $\abs{\cdot}_p$ and let $\Q_\infty = \R$. We consider the \textit{$S$-arithmetic numbers} given by $\Q_S = \prod_{p\in S} \Q_p$. We denote an element in $\Q_S$ by $\sx = (x_p)_{p\in S}$, and when clear use $|x_p|_p = |\sx|_p$ interchangeably. To distinguish the case when the element is given by the diagonal embedding into $\Q_S$, given $z\in \Q$, we will use the same notation of $z \in \Q_S$ for the element $(z)_{p\in S}$. The corresponding \textit{ring of $S$-integers} is given by
$$\Z_S = \{z \in \Q_S: z\in \Q \text{ and } |z|_p \leq 1 \text{ for all } p\notin S\} = \{z \in \Q_S: z\in \Z[p_1^{-1},\ldots, p_s^{-1}]\}.$$ For convenience, we will also denote $\Z_S = \Z[p_1^{-1},\ldots, p_s^{-1}] \subset \Q$ without the diagonal embedding and any element $z\in \Z_S$ will be denoted as such for both the element of $\Q$ and the element of $\Q_S$ under the diagonal embedding. When $S = \{\infty\}$ we recover $\Q_S = \R$ and $\Z_S = \Z$.

\begin{notation}[$S$-arithmetic numbers]\label{S-arithmetic numbers} For  $S=\{\infty, p_1, \ldots, p_s\}$,
\begin{enumerate}
\item $\Z_S^\times = \{\pm p_1^{k_1}\cdots p_s^{k_s}: k_1,\ldots, k_s \in \Z\}$ is the set of units in $\Z_S$, and we identify $\Z_S^\times$ with its diagonal embedding in $\Q_S$; 
\item $\N_S=\{m \in \N : \gcd(m,p)=1\text{ for all } p\in S_f\}$;
\item $L_S=\prod_{p\in S_f} L_p$, where for each $p\in S_f$, set $L_p=p$ if $p\neq 2$ and $L_2=2^3$;
\item $\zeta_S(d)=\sum_{m\in \N_S} \frac 1 {m^d}$ is the $S$-arithmetic zeta function {at $d$} for each $d\in \N_{\geq 2}$;
\item $\rd(\sx)=\prod_{p\in S} |x_p|_p$ for invertible $\sx=(x_p)_{p\in S}\in \Q_S$. :
\end{enumerate}
\end{notation}

When $S=\{\infty\}$, we have $\mathbb{N}_S = \mathbb{N}$, $L_S = 1$, $\zeta_S$ is the classical Riemann zeta function, and $\rd$ is the absolute value function.
We denote an element of the product space $\bsv \in \Q_S^d$ by $\bsv = (\bv_p)_{p\in S}$, where each $\bv_p \in \Q_p^d$. The volume measure $\vol_S$ on $\Q_S^d$ is the product of the usual Lebesgue measure $\vol_\infty$ on $\R^d$ and the normalized Haar measure $\vol_p$ on $\Q_p^d$, $p<\infty$, for which $\vol_p(\Z_p^d)=1$.

\begin{notation}[$S$-arithmetic groups] We set
\begin{enumerate}
	\item $\GL_d(\Q_S)=\prod_{p\in S} \GL_d(\Q_p) = \prod_{p\in S} \{d\times d \text{ matrices over } \Q_p \text{ with nonzero determinant}\}$;
	\item $\SG_d = \SL_d(\Q_S) = \prod_{p\in S} \SL_d(\Q_p) = \prod_{p\in S}\{g_p\in \GL_d(\Q_p) : \det g_p=1\}$;
	\item $\Gamma_d = \SL_d(\Z_S)$ is the set of determinant $1$ matrices with entries in $\Z_S \subset \Q$. We use the same notation for $\Gamma_d$ under the diagonal embedding into $\SG_d$.
	\end{enumerate}	
\end{notation}

\begin{rmk}
	Note that one might naively expect $\Gamma_d$ to be given by $\SL_d(\Z)\times\prod_{s\in S_f} \SL_d(\Z_p)$, but in fact for $p \in S_f$ the space $\SL_d(\Z_p)$ acts as a fundamental domain in the quotient space. That is $\SG_d/\Gamma_d$ has a fundamental domain given as the product of a fundamental domain of $\SL_d(\R)/\SL_d(\Z)$ and $\SL_d(\Z_p)$ for each $p\in S_f$.
	\end{rmk}
Let $\mu_d$ be the normalized Haar measure on $\SG_d$ for which $\mu_d\left(\SG_d/\Gamma_d\right)=1$. When $d=2$, we also consider the measure $\eta_2$ on $\SG_2$ defined as follows: for generic $\sg \in \SG_2$, it can be decomposed by
\[
\sg=
\bpm
1 & 0 \\
\sc & 1 \\ 
\epm
\bpm
\sa & \sb \\
0 & \sa^{-1} \\ 
\epm.
\]
Then $d\eta_2(\sg)=d\sa \, d\sb\,  d\sc$. 
One can check that $\eta_2$ is a Haar measure and $\mu_2=\frac 1 {\zeta_S(2)}\eta_2$ (\cite{GH21}).

\subsection{The space of unimodular $S$-lattices} \label{sec:unimodular}
An \emph{($S$-)lattice} $\Lambda$ in $\Q_S^d$ is defined as a free $\Z_S$-module in $\Q_S^d$ of rank $d$. That is, there are $\bsv^1, \ldots, \bsv^d\in \Q_S^d$ such that their $\Z_S$-span is $\Lambda$ and $\Q_S$-span is $\Q_S^d$. Denote by $\rd(\Lambda)$ the covolume of $\Lambda$ with respect to $\vol_S$. We say that $\Lambda$ is \textit{unimodular} if $\rd(\Lambda) = 1$.

The group $\GL_d(\Q_S)$ acts linearly for each component in the product space $\Q_S^d$. Namely for $\sg=(g_p)_{p\in S} \in \GL_d(\Q_S)$ and $\bsv=(\bv_p)_{p\in S} \in \Q_S^d$, the action of $\sg$ at $\bsv$ is given by $\sg\bsv=(g_p\bv_p)_{p\in S}$. From this action, one can deduce that  $\sg\Z_S^d$ for $\sg\in \GL_d(\Q_S)$, is a lattice with covolume
\[
\rd(\sg\Z_S^d)=\prod_{p\in S} |\det g_p|_p.
\]
Notice that the definition of the covolume $\rd$ coincides with \cref{S-arithmetic numbers} (5), as $\rd(\sx)$ for invertible $\sx \in \Q_S$ is the covolume of the lattice $\sx \Z_S$ in $\Q_S$.
For $p\in S_f$, one can consider the group
\[
\UL_d(\Q_p)=\{g_p\in \GL_d(\Q_p) : |\det g_p|_p=1\}
\]
which is an open subgroup in $\GL_d(\Q_p)$. Denote $$\UL_d(\Q_S)=\SL_d(\R)\times \prod_{p\in S_f} \UL_d(\Q_p).$$

It is known that the space of unimodular lattices in $\Q_S^d$ is identified with $\UL_d(\Q_S)/\UL_d(\Z_S)$ and $\SG_d/\Gamma_d$ is a proper subspace of the space of unimodular lattices.
In this paper, we concentrate our attention on $\SG_d/\G_d$ since
the primitive integral formulas over unimodular lattices are easily deduced from the proofs of those for $\SG_d/\Gamma_d$. Moreover, applications for $\UL_d(\Q_S)/\UL_d(\Z_S)$ can be obtained from the results on $\SG_d/\Gamma_d$ by integrating on variables related to $(\det g_p)_{p\in S_f}$.

\subsection{Primitive vectors and the primitive Siegel transform} \label{sec:primvectors}
The primitive vectors in $\Z_S^d$ are defined by 
$$P(\Z_S^d) = \Gamma_d\cdot {\bf e_1},$$
where we again use the identification of ${\bf e_1} = {\tp{(1,0,\ldots, 0)}}\in \Z^d$ with the diagonally embedded element ${\bf e_1} = ({\bf e_1})_{p\in S}$. Notice that when $S= \{\infty\}$, we recover the primitive integer lattice given by all points in $\Z^d$ which do not have a common factor: $P(\Z^d) = \SL_d(\Z)\cdot {\bf e}_1$. 

We now state two equivalent characterizations of the primitive $S$-arithmetic vectors. The first identifies the connection between the $S$-primitive lattice in $\Z_S^d$ and the integer primitive lattice in $\Z^d$. This fact was used in \cite{GH21}, but we state a proof here for completeness. The second characterization reflects the fact that $P(\Z^d)$ are exactly the elements in $\bv \in \Z^d$ with $\gcd(\bv) = 1$.

\begin{prop}\label{relation btwn primitive and S-primitive}  Identifying $P(\Z^d)$ with its image under the diagonal embedding in $\Z_S^d$, 
\[
P(\Z_S^d)=\Z_S^\times\cdot P(\Z^d).
\]
\end{prop}
\begin{proof} 
	 We consider the sets before the diagonal embedding.  If $\bv \in P(\Z_S^d)$, then $\bv = g {\bf e_1}$ for some $g\in \Gamma_d$. Since the entries of $g$ live in $\Z_S = \Z[p_1^{-1},\ldots, p_s^{-1}]$, choose appropriate integers $k_1,\ldots, k_s$ (including $0$) so that $\widetilde{{\bv}} = p_1^{k_1}\cdots p_s^{k_s} g {\bf e_1} \in \Z^d$ and $\widetilde{\bv}/p \notin \Z^d$ for any $p\in S_f$. Notice that $\gcd(\widetilde{\bv})\in \N_S$ by our choice of $k_1, \ldots, k_s$. Suppose that $\gcd(\widetilde \bv)= m \ge 2$. Then $\det g\in m\Z_S$ since $\bv=g{\bf e_1}\in m\Z_S^d$ is the first column of $g$. This contradicts the fact that $\det g=1$.

	In the reverse direction let $ p_1^{k_1}\cdots p_s^{k_s} \in \Z_S^\times$ for $k_1,\ldots ,k_s \in \Z$ and let $\bv  \in P(\Z^d)$. Then $\bv = g {\bf e_1}$ for some $g \in \SL_d(\Z)$. Now consider the matrix 
	$$\widetilde{g} = g\;\diag( p_1^{k_1}\cdots p_s^{k_s},p_1^{-k_1}\cdots p_s^{-k_s}, 1, \ldots, 1).
	$$
	 Then $\widetilde{g} {\bf e_1}= p_1^{k_1}\cdots p_s^{k_s} \bv $, and moreover $\widetilde{g} \in \Gamma_d$ since $\det \widetilde{g}=1$ and the entries of $\widetilde{g}$ live in $\Z_S$. 
\end{proof}

\begin{defn}\label{def:Sgcd}
   The \emph{$S$-greatest common divisor} $\sgcd(\bv)$ of a vector $\bv\in \Z_S^d$, which takes a value in $\N_S$ is given as follows. For a given $\bv\in \Z_S^d$, let $k_1, \ldots, k_s$ be the smallest integers in $\N\cup\{0\}$ for which $\bv'=p_1^{k_1} \cdots p_s^{k_s}\bv\in \Z^d$. Denote $\gcd(\bv')=p_1^{k'_1}\cdots p_s^{k'_s}m$, where $k'_1,\ldots, k'_s\in \N\cup\{0\}$ and $m\in \N_S$. We define $\sgcd(\bv)=m.$
\end{defn}


\begin{prop} \label{prop:sgcd} The primitive vectors are exactly those with an $S$-greatest common divisor of $1$:
	 $$P(\Z_S^d) = \{\bv \in \Z_S^d:  \sgcd(\bv) = 1\}.$$
\end{prop}
\begin{proof}
The result follows almost directly from \cref{relation btwn primitive and S-primitive} and the definition of $\sgcd$.
\end{proof}


\subsection{Mean values for the primitive and non-primitive Siegel transforms} \label{sec:meanvalues}
For $f: \Q_S^d \to \R$, define the \textit{$S$-primitive Siegel transform} by 
\begin{equation}\label{def: S-primitive Siegel transform}
\widehat{f}(\sg\Gamma_d) = \sum_{\bv \in P(\Z_S^d)} f(\sg \bv),
\end{equation}
and the \textit{$S$-Siegel transform by}
$$\widetilde{f}(\sg\Gamma_d) = \sum_{\bv \in \Z_S^d-\{\origin\}} f(\sg\bv)$$
for $\sg\Gamma_d \in \SG_d/\Gamma_d$.
More generally, Siegel transforms can be defined over the space of lattices in $\Q_S^d$.

For integrability criterion, we work with bounded functions of compact support, denoted $B_c(X)$. The space of semicontinuous functions which are bounded and of compact support is denoted by $B_c^{SC}(X)$.

\begin{rmk}\label{rmk: range of ftns}
We write the set of \textit{semicontinuous} real-valued functions on a space $X$, as $SC(X)$. Note that $f\in SC(X)$ is either upper semicontinuous or lower semicontinuous. Recall a function $f$ is upper (resp. lower) semicontinuous at a point $x_0 \in X$ if $\limsup_{x\to x_0} f(x) \leq f(x_0)$ (resp. $\liminf_{x\to x_0} f(x) \geq f(x_0)$). 
 Extending the class of functions beyond the standard continuous functions of compact support is useful since $SC(X)$ contains all characteristic functions of sets that are either open or closed. 

Also, though we will only work with real-valued functions, each of the integral formulas can be written for complex-valued functions by considering the real and imaginary parts separately.
\end{rmk}

In \cite{HLM2017} (c.f. \cite[Proposition~2.3]{Han21}), for any $f\in B_c^{SC}(\Q_S^d)$, $d\ge 2$, they show
\begin{equation} \label{eq:meanvalue}
\int_{\SG_d/\Gamma_d} \widetilde{f}(\sg\Gamma_d) \,d\mu_d(\sg) = \int_{\Q_S^d} f(\bsx)\,d\bsx,
\end{equation}
where $d\bsx=d\vol_S(\bsx)$. \Cref{prop:primmean} in the introduction is the primitive version of the above integral formula.
The proof is contained in \cref{A mean value formula} and uses \cref{lem:Reiszrep} stated in the next section.

  We conclude the discussion on mean values, and transition to the second moment formulas by discussing boundedness and connections to integrability. In the case of $d\geq 3$ the real case and $S$-arithmetic case are similar in the sense that $\widehat f$ is unbounded for any $f\in B_c^{SC}(\Q_S^d)$ for which $\supp (f)$ has an open interior. However when $d=2$ the behavior of $\widehat{f}$ is drastically different when $S = \{\infty\}$ versus having at least one prime included. Indeed when $d= 2$ and $S=\{\infty\}$, $\widehat f$ is bounded for any $f\in B_c^{SC}(\R^2)$ \cite[Theorem~16.1]{Veech98}, so now integrability of $\widehat{f}$ and higher moments are a direct consequence. However, when $S$ includes at least one prime, $\widehat{f}$ can be unbounded. Namely set $S=\{\infty,p\}$ and let $f\in B_c^{SC}(\Q_S^2)$ be the product of the characteristic functions of the closed ball of radius $1$. 
 Set for each $k\in \N$,
 \[
 \sg_k:=\left(\left(\begin{array}{cc}
 1/p^k & 0\\
 0 & p^k\end{array}\right), \left(\begin{array}{cc}
 1 & 0 \\
 0 & 1 \end{array}\right)\right).
 \]
 Then for $1\le \ell \le k$, 
 \[
 \sg_k(p^\ell \be_1)=\left(\left(\begin{array}{c}
 p^{-k+\ell} \\
 0\end{array}\right), \left(\begin{array}{c} 
 p^\ell \\
 0\end{array}\right)\right)\in \supp(f),
 \]
 so that $\widehat{f}(\sg_k\Gamma_2)\ge k$ and $\widehat{f}(\sg_k\Gamma_2)$ diverges to infinity as $k$ goes to infinity.

\subsection{Second moment primitive mean value formulas}\label{sec:secondmomentformulas}
In order to understand higher moments, we will consider the \textit{higher $S$-primitive Siegel transform} defined for $k\geq 1$ and $F: (\Q_S^d)^k \to \R$ by 
$$\widehat{F}(\sg\Gamma_d) = \sum_{(\bv^1,\ldots, \bv^k) \in P(\Z_S^d)^k} F(\sg\bv^1,\ldots, \sg \bv^k).$$
We will use the same notation for higher moments, as the definitions are determined by the domains of functions specified in each theorem statement.

We now give a representation theorem for the primitive Siegel transform and higher $S$-primitive Siegel transforms. To understand the distintion in the integrability criterion, we recall the case of higher moments of $\widetilde{f}$ which give upper bounds for $\widehat{f}$. Namely we have integrability from the mean value of $(\widetilde{f})^k$ as in \cite[Theorem~2.5]{Han21} for $d\geq 3$ and $1\leq k\leq d-1$. The case of $d=2$ is different as $(\widetilde{f})^2$ is not integrable even in the case of $\SL_2(\mathbb{R})/ \SL_2(\Z)$ \cite{EMM98}.
\begin{lem}\label{lem:Reiszrep}
    Let $d\geq 2$. There exists a unique regular $\SG_d$-invariant Borel measure $\nu$ on $\Q_S^d$ such that for $f \in B_c^{SC}(\Q_S^d)$,
    $$\int_{\SG_d/\Gamma_d} \widehat{f}(\sg\Gamma_d)\,d\mu_d(\sg) =  \int_{
    \Q_S^d} f(\bsx)\,d\nu(\bsx).$$
    For $d\geq 3$ and $k\leq d-1$, there exists a unique regular $\SG_d$-invariant Borel measure $\nu_k$ on $(\Q_S^d)^k$ such that for $F \in B_c^{SC}\left((\Q_S^d)^k\right)$,
    $$\int_{\SG_d/\Gamma_d} \widehat{F}(\sg\Gamma_d)\,d\mu_d(\sg) =  \int_{
    \Q_S^d} {F(\bsx_1, \ldots, \bsx_k)\,d\nu_k(\bsx_1, \ldots, \bsx_k)}.$$
\end{lem}
\begin{proof}
     We will outline the standard approximation arguments needed for the first result for $d\geq 2$, and the second result follows identically since in all cases integrability is automatic using integrability in the non-primitive setting. 
     For $f \in C_c(\Q_S^d)$ we have {$|\widehat{f}|\leq |\widetilde{f}|$}, which is integrable. Thus $f\mapsto \int_{\SG_d/\Gamma_d} \widehat{f}\,d\mu_d(\sg)$ defines an $\SG_d$-invariant positive linear functional, implying by the Riesz--Markov--Kakutani theorem that there is a unique Borel measure $\nu$ where the integral formula holds for all $f \in C_c(\Q_S^d)$. Since every lower semicontinuous function with compact support bounded below can be approximated by a non-decreasing sequence $f_n \in C_c(\Q_S^d)$ converging pointwise to $f$ and moreover we have pointwise monotone convergence of $\widehat{f_n}$ to $\widehat{f}$, we can apply the monotone convergence theorem on each side of the representation. Similarly by taking the negative, we can extend the formula using the monotone convergence theorem for upper semicontinuous functions bounded from above. Thus the integral formula in fact holds by monotone convergence theorem for all $f\in B_c^{SC}(\Q_S^d)$. 
\end{proof}

We use the representation theorem in the case when $d\geq 3$, and obtain a formula for the second moment which is stated in \cref{thm:primrog} and is proved in \cref{sec:primrogers}. When $d\geq 4$ we know by \cref{lem:Reiszrep} that $\widehat{F}$ is integrable for $3\leq k \leq d-1$ and is represented by some measure $\nu_k$, so we could theoretically find formulas for $k\geq 3$, but in this case the possible sets invariant under the diagonal action of $\Gamma_d$ are numerous. Thus we will focus on the case when $k=2$.

The case when $d=2$ must be treated differently. The main reason is that $\SL_d(\Q_S)$ acts transitively on the nonzero points of $\Q_S^d \times \Q_S^d$ for $d\geq 3$, but when $d=2$ the action is no longer transitive with orbits restricted to {subsets of $\Q_S^2 \times \Q_S^2$ with fixed determinants.}

\begin{thm}[Primitive $S$-arithmetic second moment for $d=2$]\label{rank 2 1st form}
	For $F \in B_c^{SC}( \Q_S^2 \times \Q_S^2)$ with $F\geq 0$, it holds that $\widehat{F} \in L^1(\SG_2/\Gamma_2)$ and
	$$\int_{\SG_2/\Gamma_2} \widehat{F}(\sg\Gamma_2) \, d\mu_2(\sg) = \sum_{n\in \Z_S-\{0\}} \frac{\varphi({\rd(n)})}{\zeta_S(2)} \int_{\SG_2} F(\sg J_n) \, d\eta_2(\sg) + {\frac 1 {\zeta_S(2)}\sum_{k\in \Z_S^{\times}}\int_{\Q_S^2} F(\bsx,k\bsx)\,d\bsx}, $$
	where $\varphi(\cdot)$ is Euler $\varphi$-function and $J_n=\bpm
	1 & 0 \\ 0 & n\epm$.
\end{thm}
Notice that the input of $\rd(n)$ as a positive integer into the Euler $\varphi$-function is well defined. Namely, if $n=mp_1^{k_1}\cdots p_s^{k_s}\in \Z_S-\{0\}$, where $m$ or $-m\in \N_S$ and $k_1, \ldots, k_s\in \Z$, then $\rd(n)=\prod_{p\in S} |n|_p=|m|_\infty \in \N_S$.

We prove \cref{rank 2 1st form} in \cref{sec:secondmoment d=2}. Building on \cref{rank 2 1st form}, for the applications with $d=2$, we want to compute integral formulas over \emph{the cone associated with a fundamental domain $\mathcal F\subseteq \SG_2 (\subseteq (\Q_S^2)^2)$} defined by 
\begin{equation}\label{def of cone}
\begin{split}
C_S = C_{S,\mathcal{F}} &\simeq \mathcal{F} \times \Interval_1\\
\sv^{1/2}\sg &\leftrightarrow (\sg, \sv),
\end{split}
\end{equation}
where $\Interval_1 = \Big((0,1]\times \prod_{p\in S_f} (1+L_p\Z_p)\Big)$.
Recall that $L_p=p$ if $p\neq 2$ and $L_2=2^3$.
Assign the measure $\mu_{C_S}$ on $C_S$ by the product measure $\mu_2\times \vol_S$ so that $\mu_{C_S}(C_S)=1/L_S$. In order to obtain the correct scaling factors, we will need to take the square root of elements in $C_S$. As in the real case, for odd $p$ the square root is also well-defined from $1+p\Z_p$ to $1+p\Z_p$ by Hensel's Lemma. However in the case of $p=2$, the map is well defined when we consider the image from $1+8\Z_2$ to $1+4 \Z_2$. 


\begin{defn}\label{def:phisummation}
    Define a function $\Phi_S(\sx)$, for $\sx\in \Q_S$, by
\begin{equation}\label{Phi function}
\Phi_S(\sx)=\left\{\begin{array}{cl}
 \rd(\sx)\underset{m\in \N_\sx}{\sum} \dfrac {\varphi(m)} {m^3},&\text{if }\sx\in \prod_{p\in S} (\Q_p-\{0\});\\
 0, &\text{otherwise,}
 \end{array}\right.
\end{equation}
where $\N_\sx$ for $\sx\in \prod_{p\in S} (\Q_p-\{0\})$, is the subset of $\N_S$ given by
\[
\N_\sx=
\left\{m\in \N_S : m\ge \rd(\sx)\;\text{and}\;
m\equiv 
\mathrm{sign}(x_\infty)x_p\Big(\prod_{p\in S_f}  |x_p|_p\Big)\;\mod L_p\;\;\text{for each}\; p\in S_f\right\}, 
\]

where $\mathrm{sign}(x_\infty)=x_\infty/|x_\infty|_\infty$.
\end{defn}
 
\begin{prop}[Primitive $S$-arithmetic integral formula over cone for $d=2$]\label{integral over the cone}~
Let $\SG_2=\SL_2(\Q_S)$ and $\Gamma_2=\SL_2(\Z_S)$. 
Let $C_S$ be the cone defined as in \eqref{def of cone} for some fixed fundamental domain for $\SG_2/\Gamma_2$.
We have the following. 
\begin{enumerate}
\item For $f\in B_c^{SC}(\Q_S^2)$, the function
\[(\sg, \sv) \mapsto \rd(\sv) \widehat f\left(\sv^{1/2}\sg \Gamma_2\right)
\]
is in $L^1(C_S)$ and
\[\begin{split}
\int_{C_S} \rd(\sv) \widehat f \left(\sv^{1/2}\sg \Gamma_2\right) d\mu_2(\sg) d\sv
=\frac 1 {L_S\zeta_S(2)} \int_{\Q_S^2} f(\bsx) d\bsx.\\
\end{split}\]
\item For $F \in B_c^{SC}( \Q_S^2 \times \Q_S^2)$, the function
\[
(\sg, \sv) \mapsto \rd(\sv)^2 \widehat F\left(\sv^{1/2}\sg \Gamma_2\right)
\]
is in $L^1(C_S)$ and
\[\begin{split}
&\int_{C_S} \rd(\sv)^2 \widehat F \left(\sv^{1/2}\sg \Gamma_2\right) d\mu_2(\sg) d\sv\\
&\hspace{0.2in}=\frac 1 {\zeta_S(2)}\int_{(\Q_S^2)^2} \Phi_S(\det\left(\bsx, \bsy\right))
	F\left(\bsx, \bsy\right) d\bsx d\bsy +\frac 1 {2L_S \zeta_S(2)}\sum_{k\in \Z_S^\times}\int_{\Q_S^2} F(\bsx, k\bsx) d\bsx,
\end{split}\]
where we define $\det\left(\bsx, \bsy \right)=\left(\det \left(\bx_p, \by_p\right)\right)_{p\in S}\in \Q_S$.
\end{enumerate}
\end{prop}
\begin{rmk}
Our choice of normalizing factor in the integral formula of
\[
\rd(\sv)\widehat{f}(\sv^{1/2}\sg\Gamma_2)
\]
(and hence $\rd(\sv)^2\widehat{F}(\sv^{1/2} \sg\Gamma_2)$) in \cref{integral over the cone} instead of $\widehat{f}(\sv^{1/2}\sg\Gamma_2)$ on $C_S$ comes from the following justification. Consider $f$ as the characteristic function of a Borel set $A\subseteq \Q_S^2$ of large volume. In this case, the expected value of $\widehat{f}(\Lambda)$ at the lattice $\Lambda=\sv^{1/2}\sg\Gamma_2$ is a function of $\rd(\sv)$, namely the volume of $A$ divided by the product of the covolume of $\sv^{1/2}\sg\Z_S^2$ and $\zeta_S(2)$.
Hence by multiplying $\widehat{f}$ by the covolume, one can obtain the scalar expectation value at a lattice on $C_S$ which is $\vol_S(A)/\zeta_S(2)$. This scalar expectation gives the correct scaling when changing variables from integrating over $C_S$ to $(\Q_S^2)^2$ in the second moment formula.
\end{rmk}
We will prove \cref{integral over the cone} in \cref{Integral formulas over Cone}.
Moreover, we will see in \cref{application of Abel summation formula} that $\Phi_S$ is $1/(L_S \zeta_S(2))$ up to a controlled error term. This will allow us to approximate the second moment formula in \cref{integral over the cone} (2) by
\[
\frac 1 {L_S\zeta_S(2)^2}
\int_{(\Q_S^2)^2} F(\bsx, \bsy) d\bsx d\bsy
+\frac 1 {2L_S \zeta_S(2)}\sum_{k\in \Z_S^\times}\int_{\Q_S^2} F(\bsx, k\bsx) d\bsx,
\]
which is close to the second moment formula for the higher-dimensional case in \cref{thm:primrog}. Thus we are able to use the volume of our sets for the main estimates after sufficiently controlling the error terms.

\subsection{Error terms} \label{sec:errorterms}
We now state the error terms obtained as an application of the second moment formulas in full generality. When comparing to \cref{thm:starshaped}, notice that the exponents are weaker without the additional structure of the sets.

For each $p\in S$, we consider the element $\T = (T_p)_{p\in S} \subset (\R_{\geq 0})^{s+1}$ given by $T_\infty \in \R_{\geq 0}$, and for each $p\in S_f$, $T_p \in \{p^z: z\in \Z\}$. We include a partial ordering via $\T \succeq \T'$ whenever $T_p \geq T_p'$ for all $p\in S$.

\begin{thm}\label{Schmidt Main Theorem}
Consider a collection of positive volume Borel sets $\mathcal F=\left\{A_\T \right\}_{\T= (T_p)_{p\in S}\in \mathcal T}$ so that 
\begin{enumerate}[label=(\alph*)]
\item for each $\T\in \mathcal T$, $A_\T = \prod_{p\in S} (A_\T)_{p}$ for Borel sets $(A_\T)_{p} \subseteq \Q_p^d$ with $\vol_p((A_\T)_p)=T_p$;
\item $A_{\T_1} \subseteq A_{\T_2}$ when $\T_2 \succeq \T_1$;
\item $\mathcal T$ is unbounded (as a subset of $(\R_{\ge 0})^{s+1}$), and for each $p\in S_f$, $\min_{\T\in \mathcal{T}}\{T_p\}>0.$
\end{enumerate}
We have the following two cases.
\begin{enumerate}
\item Let $d\ge 3$ and $\delta\in (\frac 2 3, 1)$. For almost all $\sg\in \SL_d(\Q_S)$
\[
\# \left(\sg P(\Z_S^d) \cap A_\T \right)
= \frac 1 {\zeta_S(d)}\vol_S(A_\T)+ O_{\sg}\left(\vol_S(A_\T)^{\delta}\right),
\]
where the dependency on $\sg$ means $\T\succeq \T_0$ for some $\T_0=\T_0(\sg)$.
\item Let $d=2$. 
Take a sequence $(\T_\ell)_{\ell\in \N}$ such that there are $\delta_1, \delta_2\in (0,1)$ such that
\[\begin{gathered}
\sum_{\ell=1}^\infty \vol_p\left((A_{\T_\ell})_p\right)^{1-2\delta_1} < \infty,\; \forall p\in S;\\
\sum_{\ell=1}^\infty \vol_S\left(A_{\T_\ell}\right)^{1+\delta'-2\delta_2}<\infty\hspace{0.1in}\text{for some}\; \delta'>0.
\end{gathered}\]
Then for almost all $\sg\in C_S$,
\[\begin{split}
&\sd(\det \sg)\#\left(\sg P(\Z_S^d) \cap A_{\T_\ell} \right) 
= \frac 1 {\zeta_S(2)}\vol_S(A_{\T_\ell})\\ 
&\hspace{0.5in}+ O\left(\sum_{p\in S} \vol_p\left((A_{\T_\ell})_p\right)^{\delta_1}\prod_{p'\in S-\{p\}} \vol_{p'}\left( (A_{\T_\ell})_{p'}\right)\right)
+ O\left(\vol_S(A_{\T_{\ell}})^{\delta_2}\right).
\end{split}\]
\end{enumerate}
\end{thm}

Notice that the convergence condition on the first summation in \cref{Schmidt Main Theorem} (2) implies that we need $\vol_p((A_{\T_{\ell}})_p)$ to diverge to infinity as $\ell$ goes to infinity for all places $p\in S$, causing the dimensional restriction in \cref{Schmidt Prim+Cong Theorem}. Moreover for $d\geq 3$, we can obtain a better exponent if we limit our Borel sets to dilates of star-shaped sets.


\begin{thm}\label{thm:starshaped}
$d\ge 3$. Let $A\subseteq \Q_S^d$ be the star-shaped Borel set given by a function $\rho=\prod_{p\in S} \rho_p$, where for each $p\in S$, $\rho_p$ is a positive function on $\{\bv_p\in \Q_p^d: \|\bv_p\|_p=1\}$. I.e., $A=\prod_{p\in S} A_p$, where
\[\begin{gathered}
A_\infty=\left\{\bv_\infty \in \R^d : \|\bv_\infty\|_\infty < \rho_\infty (\|\bv_\infty\|_\infty^{-1} \bv_\infty)\right\};\\
A_p=\left\{\bv_p \in \R^d : \|\bv_p\|_p < \rho_p (\|\bv_p\|_p \bv_p)\right\},\; p\in S_f.
\end{gathered}\] 

Consider the set $\{\T A = \prod_{p\in S} T_pA_p\}_\T$ of dilates of $A$, where $\T=(T_p)_{p\in S}$. For any $\delta\in (\frac 1 2, 1)$ and for almost all $\sg\in \SL_d(\Q_S)$,
\[
\#(\sg P(\Z_S^d)\cap \T A)=\frac 1 {\zeta_S(d)}\vol_S(\T A)+O_{ \sg}\left(\vol_S(\T A)^\delta\right),
\]
 where the dependency on $\sg$ means for all $\T\succeq \T_0$ for some $\T_0=\T_0(\sg)$
 \end{thm}


\subsection{Khintchine--Groshev Theorems}\label{sec:KG} Consider a collection $\psi=(\psi_p)_{p\in S}$ of non-increasing and non-negative functions on $\R_{>0}$ such that
\[
\psi_p\equiv 1
\;\text{on}\; (0,1]\text{ for all } p\in S.
\]
We also add a mild assumption for each finite place $p\in S_f$ that for each $k\in \Z$, there is some $\ell$ so that {$\psi_p(p^{k'}) = (p^m)^\ell$} for each $k' = kn, kn+1,\ldots, kn+(n-1)$.

We say that $\SA \in \Mat_{m,n}(\Q_S)$, an $m\times n$ matrix with entries in $\Q_S$, is \emph{$\psi$-approximable} if the system of inequalities 
\[
\| \SA  \bq + \bp\|^m_p \le \psi_p(\|\bq\|^n_p) \text{ for all } p\in S
\]
has infinitely many integer solutions $(\bp,\bq)\in \Z_S^m\times \Z_S^n$.
By \cite{Han22} the set of $\psi$-approximable matrices has measure zero in $\Mat_{m,n}(\Q_S)$ if $\int_{\Q^n_S} \prod_{p\in S}\psi_p (\|\by\|^n_p) d\by<\infty$ and the case when the integral diverges, one can obtain the quantitative Khintchine--Groshev theorem for almost all $\SA$.

In this article we state two theorems which give quantitative primitive Khintchine--Groshev theorems. The first statement is for $m+n\geq 3$, and the second addresses the case when $m+n =2$. In order to state the theorems, for $\SA\in \mathrm{Mat}_{m,n}(\Q_S)$, define the counting function
\[
\widehat{N}_{\psi, \SA}(\T)=\#\left\{(\bp,\bq)\in P(\Z_S^{m+n}): \begin{array}{c}
\|\SA\bq+\bp\|_p^m \le \psi_p(\|\bq\|_p^n);\\
\|\bq\|_p^n \le T_p, \end{array}\; \forall p\in S \right\},\]
{where $\T=(T_p)_{p\in S}\in (\R_{>0})^{s+1}$,}
and define the volume normalization
\[
V_{\psi}(\T)= {2^{m}}\int_{\{\by\in \Q_S^n: \|\by\|_p\le T_p,\;\forall p\in S\}} \prod_{p\in S} \psi_p (\|\by\|_p^n)d\by.
\]

When taking the limit for \emph{times} $\T$, we can either restrict our sequence of $\T$ to a subsequence, or allow for any sequence of $\T$ with the following additional assumption.
We say that $\psi$ has the \textit{bounded extremal times property} if there are $\delta_1,\;\delta_2>0$ with $\delta_1+1<\delta_2<\delta_1+3$ and $C=C(\psi, \delta_1, \delta_2)>0$ such that
\begin{equation}\label{eq:KGextra}
\#\left\{\T\in (\R_{\ge 1}\cup\{\infty\})\times \prod_{p\in S_f} \{p^z: z\in \N\cup\{0,\infty\}\}\}:
\begin{array}{c}
V_{\psi}(\T)\in [k^{\delta_2}, (k+1)^{\delta_2}],\;\text{and}\\
\T\;\text{is}\; (k,\delta_2)\text{-extremal}
\end{array}\right\}
<Ck^{\delta_1}
\end{equation}
for any $k\in \N$, where $\T$ is \emph{$(k,\delta_2)$-extremal} if
\[
\not\exists \T'
\;\text{s.t.}\;
V_\psi(\T')\in [k^{\delta_2}, (k+1)^{\delta_2}]
\;\text{and}\;
\begin{array}{c}
\T' \succ \T;\\
\T' \prec \T 
\end{array},\;\text{respectively}.
\]

\begin{thm}\label{primitive Khintchine-Groshev Thm} 
Let $m, n\ge 1$ be a pair of integers with $m+n\ge 3$. Let $\psi=(\psi_p)_{p\in S}$ be a collection of approximating functions described in the beginning of \cref{sec:KG}, for which 
$$\int_{\Q_S^n} \prod_{p\in S} \psi_p (\|\by\|_p^n)d\by=\infty.$$
If $\psi$ has the bounded extremal times property, then 
for almost all $\SA\in \mathrm{Mat}_{m,n}(\Q_S)$, it follows that
\begin{equation}\label{eq:KGconclusion}
\lim_{\stackrel{T_p\rightarrow \infty}{
\forall p\in S}}\frac {\widehat{N}_{\psi,\SA}(\T)}{V_{\psi}(\T)/\zeta_S(m+n)}=1.
\end{equation}

Removing the bounded extremal times property, for any subsequence $(\T_\ell)_{\ell\in \N}$ increasing with $\T_{\ell_1}\preceq \T_{\ell_2}$ for $\ell_1\le \ell_2$ and such that $\lim_{\ell\to\infty}V_{\psi}(\T_\ell)=\infty$ implies the same conclusion \cref{eq:KGconclusion} with the limit replaced by $\ell\to \infty$.
\end{thm}

The proof of \cref{primitive Khintchine-Groshev Thm} is a direct generalization of the non-primitive results. Our main contribution is in the case when $m=n=1$, where we obtain the following theorem with the same conclusion of \cref{primitive Khintchine-Groshev Thm} with different assumptions needed for subsequences of times $\T_\ell$.
\begin{thm}\label{1-d primitive Khintchine-Groshev Thm}
Let $\psi=(\psi_p)_{p\in S}$ be a collection of approximating functions described in the beginning of \cref{sec:KG}, where $\int_{\Q_p} \psi_p(y_p)dy_p=\infty$ for all $p\in S$. Suppose we have a sequence $(\T_\ell)_{\ell\in \N}$ such that there exist $\delta_1, \delta_2 \in (0,1)$ so that
 \begin{equation} \label{borel-cantelli condition}
 \begin{gathered}
\sum_{\ell=1}^\infty \vol_p\left(E_{\psi_p}(T^{(\ell)}_p)\right)^{1-2\delta_1} < \infty,\; \forall p\in S;\\
\sum_{\ell=1}^\infty \vol_S\left(E_\psi(\T_\ell)\right)^{1+\delta'-2\delta_2}<\infty\;\text{for some}\; \delta'>0,
\end{gathered}
\end{equation}
where
\[\begin{gathered}
E_{\psi_p}(T_p)=\left\{(x_p, y_p)\in \Q_p\times \Q_p : 
|x_p|_p \le \psi_p(|y_p|_p)\;\text{and}\; |y_p|_p\le T_p\;\right\};\\
E_\psi(\T)=\prod_{p\in S} E_{\psi_p} (T_p).
\end{gathered}\]
 Then for almost all $\sx\in \Q_S$,
\[
\lim_{\ell\rightarrow \infty} 
\frac {\widehat{N}_{\psi,{\sx}}(\T_\ell)}{V_\psi(\T_\ell)/\zeta_S(2)}=1.
\]
\end{thm}

When $S=\{\infty\}$, the conditions in \cref{borel-cantelli condition} are superfluous and we obtain a much simpler asymptotic result, which is a direct consequence of Schmidt's original theorem in \cite{Schmidt60a}.

\begin{coro}\label{1-d primitive Khintchine-Groshev Thm real} Let $\psi:\R_{>0}\rightarrow \R_{\ge0}$ be a non-increasing function for which {$\sum_{q=1}^{\infty} \psi(q)$} diverges. For $x\in \R$, define
\[\begin{gathered}
N_{\psi,x}(T)=\#\left\{\frac p q \in \Q : \left| x - \frac p q \right|< \frac {\psi(q)} q \;\text{and}\; 1\le q < T\right\}.
\end{gathered}\]
Then for almost all $x\in \R$, 
\[
\lim_{T\rightarrow \infty} \frac {N_{\psi, x}(T)} {2\sum_{1\le q \le T} \psi (q)/\zeta(2)}=1.
\]
\end{coro}

Applying the same argument in the proof of \cref{1-d primitive Khintchine-Groshev Thm real} with different domains, for instance,
\[
(x,y)\in \R^2: \left\{\begin{array}{c}
0< x \le \psi(y)\;\;\text{or} \\[0.05in]
- \psi(y) \le x  < 0,\;\text{resp.}
\end{array}\right.
\quad\text{and}\quad
1\le y < T,
\]
one can obtain that for almost all $x\in \R$, the number of rationals $p/q\in N_{\psi,x}(T)$ for which $x-p/q>0$ and $x-p/q<0$ respectively, are asymptotically equal, which is $\sum_{1\le q\le T} \psi(q)/\zeta(2)$.

\subsection{Logarithm laws}\label{sec:LL}

The logarithm law for a unipotent flow given in \cref{thm:loglawAM} can be verified by giving an upper bound and then a lower bound, similar to \cite{AM09}.

\begin{lem}\label{upper bound of AM}
For $d\ge 2$, it follows that for $\mu_d$-almost every $\Lambda$, where $\Lambda=\sg\Z_S^d$ or $\Lambda=\sg P(\Z_S^d)$ for $\sg\in \SG_d$,
\[
\limsup_{|\sx|\rightarrow\infty}
\frac{\log(\alpha_1(\su_\sx \Lambda))}{\log\left(\prod_{p\in S} |x_p|_p\right)} \le \frac{1}{d}
\]
\end{lem}

\begin{lem}\label{lower bound of AM}
	Fix $d\geq 2$. Let $\{\su_\sx: \sx\in \Q_S\}$ be the one-$\Q_S$-parameter subgroup of $\SG_d$. For $\mu_d$-almost every $\Lambda$, where {$\Lambda = \sg\Z_S^d$ (for $d\geq 3$) or $\Lambda = \sg P(\Z_S^d)$ (for $d\geq 2$) for some $\sg\in \SG_d$,}  
	$$\limsup_{|\sx|\to\infty} \frac{\log(\alpha_1(\su_\sx \Lambda))}{\log\left(\prod_{p\in S} |x_p|_p\right)} \geq \frac{1}{d}.$$
\end{lem}

The most technical part of the proof is constructing a family of sets which gives the desired lower bound. In order to prove the lower bound, we will make use of an $S$-arithmetic random Minkowski theorem analogous to \cite[Theorem 2.2]{AM09}. The idea is to bound the probability that a lattice will avoid a set in terms of the volume of the set, capturing the intuitive idea that large sets are harder to avoid than small sets.
\begin{prop}[Random Minkowski]\label{prop:random minkowski}
	There is a constant $C'_d>0$ so that if $A = \prod_{p\in S} A_p$, where each $A_p \subseteq \Q_p^d$ is a measurable subset with $\mu_d(A) >0$, then
	\[\begin{split}
	\mu_d\left(\left\{\Lambda \in \SG_d/\Gamma_d: (\Lambda-\{\origin\}) \cap A =\emptyset\right\}\right)
	&\leq \mu_d\left(\left\{\Lambda \in \SG_d/\Gamma_d: P(\Lambda) \cap A =\emptyset\right\}\right) \\
	&\hspace{-1.5in}\leq \begin{cases} \dfrac{C'_d}{\vol_S(A)} &\text{when } d \geq 3,\\[0.15in]
	\dfrac{{ C'_d} E(A)}{\vol_S(A)} &\text{when } d = 2.
	\end{cases}
		\end{split}\]
	 Here for $d=2$, we define  
 $$E(A) = \left( (\log\vol_S(A))^{2+s} +\left[\sum_{p\in S}\prod_{p'\in S- \{p\}} \vol_{p'}(A_{p'})\right] \right),$$
 and we additionally need $\frac {\vol_S(A)}{(\log\vol_S(A))^{1+s}}>r_0$, where $r_0$ is given by \cref{var upper bound} (3).
\end{prop}
The proof of the lower bound will then use the following corollary of \cref{prop:random minkowski} {and $s=\# S_f$}.

\begin{coro}\label{cor:limit rand minkowski}
Let $\{A_k=\prod_{p\in S} A_k^{(p)}\}_{k\in\N}$ be a sequence of $\Q_S^d$ for which
\begin{itemize}
\item ($d\geq 3$) $\vol_S(A_k)\rightarrow \infty$ as $k\rightarrow \infty$; 
\item ($d= 2$) $\vol_p(A_k^{(p)})\to \infty$ as $k\rightarrow \infty$ for all $p\in S$.
\end{itemize} 
Then
	$$\lim_{k\to\infty}\mu_d(\{\sg\Gamma_d\in \SG_d/\Gamma_d: \sg P(\Z_S^d) \cap {A_k} = \emptyset\}) = 0.$$
\end{coro}

\begin{proof}[Proof of \cref{cor:limit rand minkowski}] 
	This follows directly when $d\geq 3$, and when $d=2$, we notice that $\vol_S(A)$ grows faster than $E(A)$, so the upper bound tends to zero in the limit.
\end{proof}

\section{Proofs of Primitive integral formulas}\label{Proofs of Primitive integral formulas} 
In this section, we start with \cref{A mean value formula} where we prove the mean value theorem for primitive $S$-arithmetic lattices (\cref{prop:primmean}). \cref{sec:primrogers} proves the primitive second moment for $d\geq 3$ as stated in \cref{thm:primrog}. The rest {of} the section is devoted to \cref{sec:secondmoment d=2} where we prove \cref{rank 2 1st form} in two parts, giving the integral formula first, and then later proving integrability.

\subsection{A mean value formula}\label{A mean value formula}

Our goal is to prove the mean value theorem of \cref{prop:primmean}. 
\begin{proof}[Proof of \cref{prop:primmean}]
	Recalling \cref{eq:meanvalue}, since $\widehat{f} \leq \widetilde{f}$, we know that $\widehat{f}$ is also integrable for $d\geq 2$. To calculate the integral formula we first closely follow the proof of \cite[Proposition 3.11]{HLM2017}. Notice that the map $f\mapsto \int_{\SG_d/\Gamma_d}\widehat{f} d\mu_d(\sg)$ is a $\SG_d$-invariant linear functional and thus by \cref{lem:Reiszrep}, is given by a linear combination of product measures $\otimes_{p\in S} \nu_p$, where each $\nu_p$ is either the Haar measure $\vol_p$ or the delta measure at zero, say $\delta_p$. Since the $\SG_d$-orbit of the set $P(\Z_S^d)$ excludes points containing zero in $\Q_p^d$ for any $p\in S$, 
	as in the proof of {\cite[Lemma 3.12]{HLM2017}}, the only possible measure with nonzero coefficient in the linear combination is the product of Lebesgue measures, which is exactly the measure $\vol_S$ which we consider on $\Q_S^d$. Thus there is a positive constant $c>0$ so that
	\begin{equation}\label{eq:representation}\int_{\SG_d/\Gamma_d} \widehat{f}(\sg) \,d\mu_d(\sg) = c \int_{\Q_S^d} f(\bsx) \,d\bsx.
 \end{equation}

We decompose $\Z_S^d- \{\origin\}$ into subsets determined by the $\sgcd$ (\cref{def:Sgcd}) to obtain
\begin{equation}\label{eq:primdecomp} \Z_S^d-\{\origin\} = \bigsqcup_{\ell \in \N_S} \ell P(\Z_S^d)
\;\Rightarrow\;
\widetilde{f}(\sg\Z_S^d)=\sum_{\ell\in \N_S} \widehat{f_\ell} (\sg\Z_S^d),
\end{equation}
where $f_\ell(\cdot) = f(\ell\cdot)$ and $\N_S$ is defined in \cref{S-arithmetic numbers}. We compute by \eqref{eq:meanvalue}, \eqref{eq:primdecomp}, and \eqref{eq:representation}
$$\int_{\Q_S^d} f(\bsx)\,d\bsx = \int_{\SG_d/\Gamma_d} \widetilde{f} \,d\mu_d = \sum_{\ell \in \N_S} \int_{\SG_d/\Gamma_d} \widehat{f_\ell} \,d\mu_d = \sum_{\ell \in \N_S} c \int_{\Q_S^d} f_\ell(\bsx)\,d\bsx = \sum_{\ell \in \N_S} \frac{c}{\ell^d}\int_{\Q_S^d} f(\bsx)\,d\bsx,$$
where in the last equality we use that the Jacobian of the mapping $\bsx\mapsto \ell \bsx$ is the product of the Jacobians on each component of the product space, which is $\frac{1}{\ell^d}$ on $\R^d$, and 1 on $\Q_p^d$ for $p\in S_f$ since $\ell \in \N_S$ is a unit of $\Q_p$ and thus preserves volume. Thus comparing coefficients we have now shown $1 = c \sum_{\ell \in \N_S} \frac{1}{\ell^d} = c \zeta_S(d),$ as desired.
\end{proof}

Now we will obtain the second moment formula for the $S$-primitive Siegel transform using different methods for $d\ge 3$ and $d=2$, respectively. As a result, the integral formula for the $2$-dimensional case looks very different from {the integral formula} for the higher dimensional case, as already known as in \cite{Rogers55, Schmidt60} for the real case.

\subsection{Primitive second moment formula for $d\geq 3$} \label{sec:primrogers}
One can obtain \cref{thm:primrog} by applying the similar strategy used in the proof of \cref{prop:primmean}, following the ideas of \cite{Han21}. 

\begin{proof}[Proof of \cref{thm:primrog}]
Since $F\in B^{SC}_c((\Q_S^d)^2)$ has compact support, we can bound $F(\bsx,\bsy) \leq f(\bsx)f(\bsy)$ for some function $f \in B_c^{SC}(\Q_S^d)$, and so by \cite[Theorem~2.5]{Han21} $\widehat{F} \leq (\widetilde{f})^2 \in L^1(\SG_d/\Gamma_d)$. In particular, $\widehat{f}\in L^2(\SG_d/\Gamma_d)$ for any $f\in B_c^{SC}(\Q_S^d)$. 

Note that a pair $(\bv^1, \bv^2)\in P(\Z_S^d)^2$ is linearly dependent if and only if there is some $k\in \Z_S^\times$ for which $\bv^1=k\bv^2$.
Hence we have that
\[
P(\Z_S^d)\times P(\Z_S^d)
=\left\{(\bv^1, \bv^2) : \bv^1,\; \bv^2\text{ are linearly independent}\right\}\sqcup\bigsqcup_{k\in \Z_S^\times} \left\{(\bv, k\bv): \bv\in P(\Z_S^d)\right\}.
\]
Put $\LI(\Id_2)=\left\{(\bv^1, \bv^2) \in P(\Z_S^d)^2: \bv^1,\; \bv^2\text{ are linearly independent}\right\}$ and $\LD(k)= \left\{(\bv, k\bv): \bv\in P(\Z_S^d)\right\}$ for each $k\in \Z_S^\times$. By the similar argument in \cite[Section 3]{Han21}, it suffices to show the following integral formulas:
\begin{align}
\int_{\SG_d/\Gamma_d} \sum_{(\bv^1, \bv^2)\in \LI(\Id_2)} F(\sg\bv^1, \sg\bv^2) d\mu_d(\sg)
&=\frac 1 {\zeta_S(d)^2} \int_{(\Q_S^d)^2} F(\bsx, \bsy) d\bsx d\bsy; \label{eq 1:primrog}\\
\int_{\SG_d/\Gamma_d} \sum_{(\bv^1, \bv^2)\in \LD(k)} F(\sg\bv^1, \sg\bv^2) d\mu_d(\sg)
&=\frac 1 {\zeta_S(d)} \int_{\Q_S^d} F(\bsx, k\bsx) d\bsx\label{eq 2:primrog}
\end{align}
for $k\in \Z_S^\times$.

For \cref{eq 1:primrog}, in the spirit of {\cite[Step 1 in the proof of Theorem 3.1]{Han21}}, the operator on $B^{SC}_c((\Q_S^d)^2)$ given by the left-hand side of \cref{eq 1:primrog} can be expressed as the integration by a single measure on $(\Q_S^d)^2$, which comes to be the Lebesgue measure, using \cref{lem:Reiszrep}. I.e., there is a positive constant $a>0$ for which
\[
\int_{\SG_d/\Gamma_d} \sum_{(\bv^1, \bv^2)\in \LI(\Id_2)} F(\sg\bv^1, \sg\bv^2) d\mu_d(\sg)
=a \int_{(\Q_S^d)^2} F(\bsx, \bsy) d\bsx d\bsy.
\]

Since
\[\begin{split}
&\left\{(\bv^1, \bv^2)\in (\Z_S^d)^2 : \bv^1,\; \bv^2\;\text{are linearly independent}\right\}\\
&=\bigsqcup_{\ell_1, \ell_2\in \N_S}
\left\{(\ell_1 \bw^1, \ell_2 \bw^2) : \bw^1, \bw^2\in P(\Z_S^d)\;\text{are linearly independent}\right\},
\end{split}\]
by considering functions $F_{\ell_1, \ell_2}(\bv^1, \bv^2):=F(\ell_1\bv^1, \ell_2\bv^2)$ for each $(\ell_1, \ell_2)\in \N_S^2$ and applying \cite[Theorem 3.1]{Han21}, it follows that
\[\begin{split}
&\int_{(\Q_S^d)^2} F(\bsx, \bsy)d\bsx d\bsy
=\int_{\SG_d/\Gamma_d} \sum_{\scriptsize \begin{array}{c}
\bv^1, \bv^2\in \Z_S^d\\
\text{lin. indep.}\end{array}} F(\sg\bv^1, \sg\bv^2) d\mu_d(\sg)\\
&=\int_{\SG_d/\Gamma_d} \sum_{\ell_1, \ell_2\in \N_S} \sum_{(\bv^1,\bv^2) \in \LI(\Id_2)} F_{\ell_1,\ell_2} (\sg\bv^1, \sg\bv^2) d\mu_d(\sg)
=\sum_{\ell_1,\ell_2\in \N_S} a \int_{(\Q_S^d)^2} F_{\ell_1,\ell_2} (\bsx, \bsy)d\bsx d\bsy\\
&=a\sum_{\ell_1,\ell_2\in \N_S} \frac 1 {\ell_1^d\cdot \ell_2^d} \int_{(\Q_S^d)^2} F(\bsx, \bsy) d\bsx d\bsy
=a \zeta_S(d)^2 \int_{(\Q_S^d)^2} F(\bsx, \bsy) d\bsx d\bsy
\end{split}\]
which shows that $a=1/\zeta_S(d)^2$.

One can obtain \cref{eq 2:primrog} by applying \cref{prop:primmean} with the function $\bsx \mapsto F(\bsx, k\bsx)$.
\end{proof}

\subsection{Primitive second moment for $d=2$}\label{sec:secondmoment d=2} We first remark that for the case when $d=2$, we don't use {\cref{lem:Reiszrep}} for the second moment formula. The principle of the formula is based on \emph{the folding and unfolding} of fundamental domains:  by considering $\SG_2\subseteq (\Q_S^2)^2$, we have that
\[
\int_{\SG_2/\Gamma_2} \sum_{\sh\in \Gamma_2}F(\sg\sh) d\mu_2(\sg)= \int_{\SG_2} F(\sg)d\mu_2(\sg)
\]
for any $F\in SC(\SG_2)$.

Recall $\LI(\Id_2)$ from the previous section. We split $\Omega(\Id_2)$ into $\Gamma_2$-orbits under the diagonal action $\gamma(\bv^1,\bv^2) = (\gamma\bv^1, \gamma \bv^2)$ for $\gamma \in \Gamma_2$. These orbits divide $\LI(\Id_2)$ by determinant, where we consider pairs $(\bv^1, \bv^2)$ as $2\times 2$ matrices. That is, $\LI(\Id_2)=\bigsqcup_{n\in \Z_S-\{0\}} D_n$, where
\[
D_n:=\left\{(\bv^1, \bv^2)\in P(\Z_S^2)\times P(\Z_S^2) :
\det\left(\bv^1, \bv^2\right)=n \right\}
\]
for each $n\in \Z_S-\{0\}$.

We first prove the integral formula in \cref{rank 2 1st form} allowing the possibility that both quantities are infinite and then show the integrability by showing the finiteness of the integral on the right hand side.

\begin{lem}\label{lemma: decomposition}
For each $n\in \Z_S-\{0\}$, $D_n$ is an $\Gamma_2$-invariant set which is the union of $\varphi(\rd(n))$ components of irreducible $\Gamma_2$-orbits, where $\varphi(\cdot)$ is the Euler totient function. In particular the representatives of the $\Gamma_2$-orbits are
\[
\bpm
1 & \ell\\
0 & n \epm \; \text{ for }\;\ell \in\left\{0,1,\ldots,\rd(n)-1: \gcd(\ell, \rd(n))=1\right\},
\]
where $\rd(\cdot)$ is defined in \cref{S-arithmetic numbers}.
\end{lem}
     Recall that $\rd(n)\in \N_S$ for $n\in \Z_S-\{0\}$. Moreover if $n\in \Z_S^\times$, then $\rd(n) = 1$ and there is exactly one $\Gamma_2$-orbit.
\begin{proof} By construction $D_n$ is $\Gamma_2$-invariant.
Let $(\bv^1, \bv^2)\in D_n$ be given. Since $\bv^1\in P(\Z_S^2)$, there is $\sg\in \Gamma_2$ for which $\sg \bv^1=\be_1$ so that
\[
\Gamma_2\left(\bv^1, \bv^2\right)
=\Gamma_2\bpm
1 & y \\
0 & n 
\epm,
\]
where ${\tp(y,n)}\in P(\Z_S^2)$. By the action of a unipotent element
\[
\bpm
1 & k \\
0 & 1 \epm \bpm
1 & y \\
0 & 1 \epm=\bpm
1 & y+kn \\
0 & n \epm,
\]
one can choose $k\in \Z_S$ such that $\ell=y+kn$ is in the fundamental domain for $\Z_S/n\Z_S $. It is easy to show that  $\Z_S/n\Z_S \simeq \Z_S/\rd(n)\Z_S\simeq \Z/\rd(n)\Z$. So the number of $\Gamma_2$-orbits in $D_n$ is the number of $\ell\in \{0,1, \ldots, \rd(n)-1\}$ such that ${\tp{(\ell,n)}}\in P(\Z_S^2)$.
By \cref{prop:sgcd}, ${\tp{(\ell,n)}}\in P(\Z_S^2)$ if and only if $\sgcd(\ell,n) = 1$, which is equivalent to the fact that $\gcd(\ell, \rd(n))=1$ by the definition of $\sgcd$.
\end{proof}

\begin{proof}[Proof of \cref{rank 2 1st form} (integral formula)]
We may assume that $F$ is non-negative so that Tonelli's theorem is applicable.
Since $D_n$ is $\SG_2$-invariant,
\begin{equation}\label{eq 3: rank 2 1st form}
\int_{\SG_2/\Gamma_2}
\widehat F(\sg\Z_S^2)\:d\mu_2(\sg)= \sum_{n\in \Z_S} \int_{\SG_2/\Gamma_2} \sum_{
(\bv^1, \bv^2)\in D_n} F(\sg\bv^1, \sg\bv^2) \:d\mu_2(\sg).
\end{equation}

We first claim that for $n\neq 0\in \Z_S$
\begin{equation}\label{eq 1: rank 2 1st form}
\int_{\SG_2/\Gamma_2} \sum_{(\bv^1,\bv^2)\in D_n} F(\sg\bv^1, \sg\bv^2) d\mu_2(\sg)
=\frac {\varphi(\rd(n))} {\zeta_S(2)}\int_{\SG_2} F(\sg J_n)d\eta_2(\sg),
\end{equation}
where $J_n=\bpm
1 & 0\\
0 & n \epm$. 


Set $m=\rd(n)$ and 
$J_{\ell,n}=\bpm
1 & \ell \\
0 & n \epm$, where $0\le \ell < m$ 
with $\gcd(\ell,m) = 1$. 
By Lemma~\ref{lemma: decomposition},
\[\begin{split}
\int_{\SG_2/\Gamma_2} \sum_{(\bv^1,\bv^2)\in D_n} F(\sg\bv^1, \sg\bv^2) d\mu_2(\sg) 
&=\sum_{\substack{
0 \le \ell < m \\
\gcd(\ell,m)=1}}\int_{\SG_2/\Gamma_2}
\sum_{(\bv^1,\bv^2)\in \Gamma_2\cdot J_{\ell,n}} F\left(\sg (\bv^1, {\bv^2})\right) d\mu_2(\sg)\\
&=\sum_{\substack{
0 \le \ell < m \\
\gcd(\ell,m)=1}}\int_{\SG_2/\Gamma_2}
\sum_{(\bv^1,\bv^2)\in \Gamma_2} F_{\ell, n}\left(\sg (\bv^1, \bv^2){J_{\ell,n}}\right) d\mu_2(\sg)\\
&=\sum_{\substack{
0 \le \ell < m \\
\gcd(\ell,m)=1}}\int_{\SG_2}
F\left(\sg J_{\ell,n}\right) d\mu_2(\sg)\\
&=\frac {\varphi(m)} {\zeta_S(2)}
\int_{\SG_2} F(\sg J_n) d\eta_2(\sg),
\end{split}\]
where we recall $\mu_2$ and $\eta_2$ are both $\SG_2$-invariant measures on $\SG_2/\Gamma_2$ with different normalization ($\mu_2(\SG_2/\Gamma_2)=1=\frac 1 {\zeta_S(2)}\eta_2(\SG_2/\Gamma_2)$) inheritted from unimodular Haar measures.
In the last line we use $\SG_2$-invariance and the change of coordinates $\sg = \sg' \begin{pmatrix} 1 & -\ell \\ 0 & 1\end{pmatrix}$.

For the rest of the proof, as in the proof for the case when $d\ge 3$, we obtain the fact that for each $k\in \Z_S^\times$,
\begin{equation}\label{eq 2: rank 2 1st form}
\int_{\SG_2/\Gamma_2} \sum_{\bv\in P(\Z_S^2)} F\left(\sg(\bv, k\bv)\right) d\mu_2(\sg)
=\frac 1 {\zeta_S(2)} \int_{\Q_S^d} F(\bsx, k\bsx)d\bsx
\end{equation}
from \cref{prop:primmean} with the function $\bsx \mapsto F(\bsx, k\bsx)$.
Therefore the integral formula follows from \eqref{eq 3: rank 2 1st form}, \eqref{eq 1: rank 2 1st form} and \eqref{eq 2: rank 2 1st form}.
\end{proof}

We now introduce \cref{thm: integrability for dim2} whose proof will complete the proof of \cref{rank 2 1st form}.
\begin{prop}\label{thm: integrability for dim2} $\widehat F\in L^1(\SG_2/\Gamma_2)$ for a non-negative function $F\in B^{SC}_c(\Q_S^2\times \Q_S^2)$.
\end{prop}
To show the result, we will first make use of the following lemma that holds for all $d\geq 2$.
\begin{lem}\label{lem:boundedunits}
	Let $d\geq 2$. Given a nonnegative $F\in B_c^{SC}(\Q_S^d \times \Q_S^d)$
 \[
		\frac 1 {\zeta_S(2)}\sum_{k\in \Z_S^{\times}}\int_{\Q_S^2} F(\bsx,k\bsx)\,d\bsx<\infty.	\]				
\end{lem}
\begin{proof}
	For the sake of simplicity, we may assume that $F=f\times f$ (i.e., $F(\bsx, \bsy)={f(\bsx)f(\bsy)}$) for some characteristic function $f$ of $A=\prod_{p\in S} A_p$, where $A_p\subseteq \Q_p^d$ is bounded, since one can always find such a function $f$ and a constant $c>0$ such that $F\le cf\times f$. Denote by $k=k_1/k_2$ if $k>0$ and $k=-k_1/k_2$ if $k<0$ for coprime $k_1,\; k_2\in \N \cap \Z_S^\times$.  Since
\[\begin{gathered}
\int_{\R^d} \one_{A_\infty}(\bv_\infty) \one_{A_\infty}(k\bv_\infty) d\bv_\infty
 \leq \min\left\{1, \frac{1}{|k|_\infty^d}\right\}\vol_\infty(A_\infty)\;\text{and}\\ \int_{\Q_p^d} \one_{A_p}(\bv_p) \one_{A_p}(k\bv_p) d\bv_p
\leq \min\left\{1, \frac{1}{|k|_p^d}\right\} \vol_p(A_p)  {=} |k_2|_p^d \vol_p(A_p),\end{gathered}\]
and since $|k|^d_\infty=k_1^d/ k_2^d$ and $\prod_{p\in S_f} |k_2|_p^d=k_2^{-d}$, it follows that for each $k\in \Z_S^\times$,
\[
\int_{\Q_S^d} f(\bsx)f(k\bsx) d\bsx 
\le \frac 1 {\max(1, |k|_\infty^d)} \vol_\infty(A_\infty)\times \frac{1}{k_2^d} \prod_{p\in S_f} \vol_p(A_p)
=\frac 1 {\max(k_1^d, k_2^d)} \vol_S(A). 
\] 
Hence it suffices to show that
\[
\sum_{\substack{
k_1,k_2\in \N \cap \Z_S^\times\\
\gcd(k_1,k_2)=1}} \frac 1 {\max(k_1, k_2)^d} < \infty
\]
and the bound depends only on the dimension $d$ and the set $S$. 
Let $\mathcal P$ be the collection of ordered pairs $(P_1, P_2)$ of partitions of $S_f$. We allow the cases when $P_1=\emptyset$ or $P_2=\emptyset$.
Then the above summation is bounded by 
\[
\sum_{(P_1,P_2)\in \mathcal P}
\sum_{k_1\in \PP_{P_1}} \sum_{k_2\in \PP_{P_2}} \frac 1 {\max(k_1, k_2)^d},
\]
where we define $\PP_{P}=\{p_{i_1}^{\ell_1}\cdots p_{i_j}^{\ell_j}: \ell_1,\ldots, \ell_j\in \N\cup\{0\}\}$ if $P=\{p_{i_1}, \ldots, p_{i_j}\}$ and $\PP_{\emptyset}=\{1\}$.

If $P_1=\emptyset$ or $P_2=\emptyset$, then the result is given by a product of geometric series
\[
\sum_{k_1\in \PP_{P_1}} \sum_{k_2\in \PP_{P_2}} \frac 1 {\max(k_1, k_2)^d}
= \sum_{\ell_1=0}^\infty \sum_{\ell_2=0}^\infty \cdots \sum_{\ell_s=0}^\infty \frac 1 {p_1^{d\ell_1} \cdots p_s^{d\ell_s}}
=\prod_{p\in S_f} \frac {p^d} {p^d-1} < \infty.
\]

 Assume $P_1\neq \emptyset$ and $P_2\neq \emptyset$. Define $q_1 = \min P_1$ and $q_2 = \min P_2$. Since the summation is symmetric without loss of generality assume $q_1 <q_2$. For $j = 1, 2$ we partition the sets $\mathbb{P}_{P_j}$ by
$$\mathbb{P}_{P_j} = {\bigcup_{M_j = 0}^\infty} \left\{k_j =  \prod_{p \in P_j} p^{\ell_p} : \sum_{p\in P_j} \ell_p =M\\
_j\right\}.$$
Thus for each fixed $M_1$ and $M_2$
\[
\frac 1 {\max(k_1, k_2)^d}
\le \frac{1}{\max( {q_1^{M_1}},  {q_2^{M_2}})^d}
\le \begin{cases}
1/{q_2^{dM_2}} &\text{if } \frac{M_1}{M_2} < \frac{\log(q_2)}{\log(q_1)};\\[0.05in]
 1/{q_1^{dM_1}}&\text{if } \frac{M_1}{M_2} \geq \frac{\log(q_2)}{\log(q_1)}.
\end{cases}
\] 
Let us divide the upper case into $M_1\leq M_2$ and $1<\frac {M_1}{M_2} < \frac{\log q_2} {\log q_1}$. Then we have
\begin{equation} \label{eq:3cases}   
\begin{split}
	\sum_{k_1\in \PP_{P_1}} \sum_{k_2\in \PP_{P_2}} \frac 1 {\max(k_1, k_2)^d}&\leq \sum_{{M_1=0}}^\infty \; \sum_{M_2= M_1}^\infty \frac{1}{q_2^{dM_2}} + \sum_{M_1=1}^\infty \sum_{M_2=  \left\lceil M_1\frac{\log(q_1)}{\log(q_2)}\right\rceil}^{M_1-1} \frac{1}{q_2^{dM_2}}\\
	&\hspace{.25in} + \sum_{M_1 =1}^\infty \sum_{M_2=0}^{M_1 \left\lceil \frac{\log(q_1)}{\log(q_2)}\right\rceil -1} \frac{1}{q_1^{dM_1}}.
\end{split}
\end{equation}
In the first part of the sum of \eqref{eq:3cases}, we use geometric series and the fact that $q_2 \geq 2$ to get
$$\sum_{M_1=0}^\infty \sum_{M_2= M_1}^\infty \frac{1}{q_2^{d{M_2}}} = \sum_{M_1=0}^\infty \frac{1}{(q_2^d)^{M_1} (1 - q_2^{-d})} \leq \sum_{M_1=0}^\infty \frac{2}{(q_2^d)^{M_1}} = \frac{2 q_2^d}{q_2^d-1}.$$
Similarly in the second part of \eqref{eq:3cases} we use the finite geometric series, the fact that $q_2 \geq 2$ and the ratio test to get
$$\sum_{M_1=1}^\infty \sum_{M_2= M_1 \left\lceil \frac{\log(q_1)}{\log(q_2)}\right\rceil}^{M_1-1} \frac{1}{q_2^{d M_2}} \leq 2  \sum_{M_1=1}^\infty \frac{q_2^{d M_1\left(1 -\left\lceil \frac{\log(q_1)}{\log(q_2)}\right\rceil\right)} -1}{q_2^{M_1d}}< \infty.$$
In the third part of of \eqref{eq:3cases} we have by the ratio test
$$\sum_{M_1 =1}^\infty \sum_{M_2=0}^{M_1 \left\lceil \frac{\log(q_1)}{\log(q_2)}\right\rceil -1} \frac{1}{q_1^{dM_1}} = \left\lceil \frac{\log(q_1)}{\log(q_2)}\right\rceil\sum_{M_1=1}^\infty \frac{M_1 }{q_1^{dM_1}} <\infty.$$

Therefore this shows the lemma, where we note all these bounds are depending only on $d$ and the set $S_f$.
\end{proof}

\begin{proof}[Proof of \cref{thm: integrability for dim2}]
By the proof of \cref{rank 2 1st form}, it suffices to show that in addition to \cref{lem:boundedunits}
\begin{align}
&\sum_{n\in \Z_S-\{0\}} \frac{\varphi(\rd(n))}{\zeta_S(2)} \int_{\SG_2} F(\sg J_n) \, d\eta_2(\sg)<\infty.\label{eq 1:finiteness dim2}
\end{align}

For \cref{eq 1:finiteness dim2}, note that the function $\det:\Q_S^2\times \Q_S^2\rightarrow \Q_S$ given by
\[
\det(\bsx, \bsy):=\left(\det(\bx_p, \by_p)\right)_{p\in S}
\]
is continuous so that $\det(\supp F)\cap \Z_S$ is finite, since $\Z_S\subset \Q_S$ is discrete and we assume that $F$ is compactly supported. Hence, the sum in \cref{eq 1:finiteness dim2} is a finite sum of finite integrals.
\end{proof}

\section{Integral formulas over $C_S$}\label{Integral formulas over Cone}

Now, let us show two integral formulas over the cone $C_S$ in Proposition~\ref{integral over the cone}. Recall that the cone $C_S=C_{S, \mathcal F}\simeq \mathcal F\times \Interval_1$, {for the fundamental domain $\mathcal{F}$ of $\SG_2/\Gamma_2$, is} defined as in \cref{def of cone}.
More generally, we will consider the cone which is parameterized by $\mathcal F\times \Interval_n$ for $n\in \Z_S-\{0\}$ in the similar way as in \cref{def of cone}, where 
\[
\Interval_n = n(0,1]\times \prod_{p\in S_f}n (1+L_p\Z_p).
\]

\begin{proof}[Proof of Proposition~\ref{integral over the cone} (1)]

We now deduce the formula from \cref{prop:primmean} and the change of variables. For each $\sv\in I_1$, set $f_{\sv}(\bsx):=f(\sv^{1/2}\bsx)$. Using Fubini's theorem,
\[\begin{split}
\int_{C_S} \rd(\sv) \widehat f \left( \sv^{1/2} \sg\Z_S^2\right) d\mu_2(\sg)d\sv
&=\int_{\Interval_1} \rd(\sv) \int_{\mathcal F} \widehat f_{\sv} (\sg\Z_S^2) d\mu_2(\sg)d\sv
=\int_{\Interval_1} \rd(\sv) \frac 1 {\zeta_S(2)} \int_{\Q_S^2} f(\sv^{1/2} \bsx) d\bsx d\sv\\
&=\int_{\Interval_1} \rd(\sv) \frac 1 {\zeta_S(2)} \int_{\Q_S^2} f(\bsx) \frac {d\bsx} {\rd(\sv)} d\sv
=\frac 1 {L_S\zeta_S(2)} \int_{\Q_S^2} f(\bsx) d\bsx.
\end{split}\]
Moreover, since the right hand side is integrable, this shows that the map $(\sv, \sg)\mapsto \rd(\sv)\widehat f \left(\sv^{1/2}\sg\Z_S^2\right)$ is in $L^1(C_S)$.
\end{proof}

For the second statement of the proposition, we first prove the integral formula regardless of finiteness and then obtain integrability by showing that the right hand side of the formula is finite, as in the proof of \cref{rank 2 1st form}. For this proof recall \cref{def:phisummation}.

\begin{proof}[Proof of \cref{integral over the cone} (2) (integral formula)]
As in the proof of \cref{integral over the cone} (1), we may assume that $F\in SC(\Q_S^2\times \Q_S^2)$ is non-negative.
For each $\sv\in \Interval_1$, define $F_\sv(\bsx, \bsy)=F(\sv^{1/2}\bsx, \sv^{1/2}\bsy)$. By Theorem~\ref{rank 2 1st form},
\[\begin{split}
&\int_{C_S} \rd(\sv)^2 \widehat F \left(\sv^{1/2}\sg\Z_S^2\right) \rd\mu_2(\sg) d\sv
=\int_{\Interval_1}\rd(\sv)^2 \int_{\mathcal F} \widehat {F_\sv}(\sg\Z_S^2) d\mu_2(\sg) d\sv\\
&=\int_{\Interval_1} \rd(\sv)^2
\left(\sum_{n\in \Z_S-\{0\}} \frac {\varphi(\rd(n))} {\zeta_S(2)}
\int_{\SG_2} F_\sv(\sg J_n) d\eta_2(\sg)+
\frac 1 {\zeta_S(2)} \sum_{k\in \Z_S^\times} \int_{\Q_S^2} F_\sv(\bsx, k\bsx)d\bsx
\right) d\sv.
\end{split}\]

First, let us compute the first part of the sum corresponding to full rank matrices. 
For each $\sv\in \Interval_1$ and $n\in \Z_S-\{0\}$, consider the change of variables $\sg'=\sg\sh_{\sv}^{-1}$, where $\sh_{\sv}=\bpm
\sv^{-1/2} & 0 \\
0 & \sv^{1/2} \epm$,
\[
\sv^{1/2}\sg J_n
=\sg' \sv^{1/2} \sh_{\sv} J_n
=\sg' \bpm
1 & 0 \\
0 & \sv n \epm.
\]
Using that $\eta_2$ is a Haar measure of $\SG_2$, we have
\begin{equation}\begin{split} \label{eq 1: integral over the cone (2)}
&\sum_{{n\in}\Z_S-\{0\}} \frac {\varphi(\rd(n))} {\zeta_S(2)}
\int_{\Interval_1} \rd(\sv)^2 \int_{\SG_2} F_{\sv} (\sg J_n)d\eta_2(\sg) d\sv\\
&=\sum_{{n\in}\Z_S-\{0\}} \frac {\varphi(\rd(n))} {\zeta_S(2)}
\int_{\Interval_1} \rd(\sv)^2 \int_{\SG_2} F\left(\sg \bpm
1 & 0 \\
0 & \sv n\epm\right) d\eta_2(\sg) d\sv.
\end{split} \end{equation}
Set $\sx=\sv n$ so that $d\sx=\rd(n)d\sv$ and $\rd(\sv)=\frac 1 {\rd(n)} \rd(\sx)$. Hence
\[
\eqref{eq 1: integral over the cone (2)} =\frac 1 {\zeta_S(2)}\sum_{{n\in}\Z_S-\{0\}} \frac {\varphi(\rd(n))} {\rd(n)^3}
\int_{\Interval_n}\rd(\sx)^2 \int_{\SG_2} F\left(\sg \bpm
1 & 0 \\
0 & \sx \epm\right) d\eta_2(\sg) d\sx.
 \]
 
 Now, we want to rearrange the above integral using Tonelli's theorem: First, we observe that for a given $\sx \in \prod_{p\in S} (\Q_p-\{0\})$,
\[
n\in \Z_S-\{0\} : \sx=(x_p)_{p\in S} \in \Interval_n
\;\Leftrightarrow\; 
x_\infty \in \left\{\begin{array}{cl}
(0, n], &\text{if } x_\infty>0;\\
{[n, 0)}, &\text{if } x_\infty <0,
\end{array}\right. 
\quad\text{and}\quad
x_p \equiv n \;\mod\; |n|_p^{-1}L_p{\Z_p}. 
\]
In particular, $|n|_p=|x_p|_p$ for $p\in S_f$.
Put $n=m p_1^{k_1}\cdots p_s^{k_s}$ if $x_\infty>0$ and $n=-mp_1^{k_1}\cdots p_s^{k_s}$ if $x_\infty<0$, where $p_i^{-k_i}=|x_{p_i}|_{p_i}$ for $1\le i \le s$ are fixed. Then the above is equivalent to the condition that
\[
m\ge |x_\infty|_\infty p_1^{-k_1}\cdots p_s^{-k_s}=\rd(\sx)
\quad\text{and}\quad
\pm\: m\equiv x_p\cdot \prod_{p\in S_f} |x_p|_p \;\mod L_p{\Z_p},\; p\in S_f	
\]
which is described as 
in \cref{def:phisummation}. 
Thus by Tonelli,
\[
\eqref{eq 1: integral over the cone (2)} =\frac 1 {\zeta_S(2)} \int_{\prod_{p\in S}   \Q_p - \{0\}} \rd(\sx)\Phi_S(\sx) \int_{\SG_2} F\left(\sg \bpm
1 & 0 \\
0 & \sx \epm\right) d\eta_2(\sg) d\sx.
 \]

Put
\[
\sg\bpm
1 & 0 \\
0 & \sx \epm
=\bpm
1 & 0\\
\sc & 1 \epm\bpm
\sa & \sb\\
0 & \sa^{-1} \epm\bpm
1 & 0\\
0 & \sx \epm
 =\bpm
\sa & \sb\sx\\
\sc\sa & (\sc\sb+\sa^{-1})\sx \epm
 =\left(\bsx, \bsy\right).
\]
{The Jacobian of the change of coordinates in each place $p\in S$ is $|\sx|_p$, thus it follows that}
$$d\eta_2(\sg)d\sx= d\sa\, d\sb\, d\sc\, d \sx =\frac 1 {\rd(\det(\bsx, \bsy))} d\bsx d \bsy,$$
where we recall $\det(\bsx, \bsy)=(\det(\bx_p, \by_p))_{p\in S}$. Hence 
\[
\eqref{eq 1: integral over the cone (2)} = \frac 1 {\zeta_S(2)} \int_{(\Q_S^2)^2}
\Phi_S\left(\det(\bsx, \bsy)\right) F(\bsx, \bsy) d\bsx d\bsy.
\]

We compute that the linearly dependent part is
\[\begin{split}
&\frac 1 {\zeta_S(2)} \int_{\Interval_1} \rd(\sv)^2 \sum_{k\in \Z_S^\times} \int_{\Q_S^2} F_{\sv} (\bsx, k\bsx) d\bsx d\sv
=\frac 1 {\zeta_S(2)} \sum_{k\in \Z_S^\times} \int_{\Interval_1} \rd(\sv)^2 \int_{\Q_S^2} F (\bsx, k\bsx) d\bsx \frac {d\sv} {\rd(\sv)}\\
&=\frac 1 {\zeta_S(2)} \sum_{k\in \Z_S^\times} \int_{\Interval_1} \rd(\sv) \int_{\Q_S^2} F(\bsx, k\bsx) 
 d\bsx d\sv
=\frac 1 {2L_S\zeta_S(2)}\sum_{k\in \Z_S^\times} \int_{\Q_S^2} F(\bsx, k\bsx) d\bsx.
\end{split}\]
Therefore we obtain the integral formula in Proposition~\ref{integral over the cone} (2).
\end{proof}

To show the integrability of \cref{integral over the cone} (2), the integral formula and \cref{lem:boundedunits} imply that it suffices to show 
\begin{equation}\label{eq 2: integral over the cone (2)}
\frac 1 {\zeta_S(2)} \int_{(\Q_S^2)^2}
\Phi_S\left(\det(\bsx, \bsy)\right) F(\bsx, \bsy) d\bsx d\bsy <\infty
\end{equation}
for a non-negative function $F\in SC(\Q_S^2\times \Q_S^2)$.
For this, we need some observations about the function
$\Phi_S$. 

Notice that for each $\sx\in \prod_{p\in S} (\Q_p-\{0\})$, there is $m_0\in\{1, \ldots, L_S -1\}$ and $\gcd(m_0, L_S)=1$ so that 
\begin{equation}\label{eq:m_0}
m\equiv x_p\Big(\prod_{p\in S_f} |x_p|_p\Big) \;\mod L_p,\; p\in S_f
\;\Leftrightarrow \;
m \equiv m_0\;\mod L_S
\end{equation}
by Sun Tzu's theorem, historically known as the Chinese remainder theorem.

Let us first show the analog of the asymptotic expansion for Euler totient summatory function \cite[Theorem 3.4]{TenenbaumIntro}, which states that 
$$\sum_{1\leq m \leq N} \varphi(m) = \frac{1}{\zeta(2)} \frac{N^2}{2} + O(N\log N),$$
where we recall $\varphi$ is the Euler totient function.

Let $\mu(\cdot)$ be the Mobius function. From the fact that
\[
\zeta_S(d)=\sum_{m\in \N_S} \frac 1 {m^d}=\prod_{\substack{q:\text{ prime}\\
\gcd(q,p_1\cdots p_s)=1}}
{\left( 1-\frac 1 {q^d}\right)^{-1}},
\]
where $d\ge 2$, using the classical properties of $\mu$ and $\zeta$, we can deduce
\begin{equation}\label{eq:summu}
\sum_{m\in \N_S} \frac {\mu(m)} {m^d} = \frac 1 {\zeta_S(d)}.
\end{equation}

\begin{lem}\label{Euler ftn-sum formula}
Let $m_0\in \N_S$ for which $1\le m_0 \le L_S-1$. For any $N\in \R_{>0}$, we have
\begin{equation*}
\sum_{\substack{
1\le m \le N \\
m\equiv m_0 \;\mod L_S}} \varphi(m)
=\frac 1 {L_S\zeta_S(2)} \frac {N^2} 2 + O_{L_S}(N\log N).
\end{equation*}
\end{lem}
\begin{proof}
It is well-known that for each $m\in \N$, $\varphi(m)=m\sum_{d|m} \mu(d)/d$. By putting $d'=m/d$, since $m_0\in \N_S$, 
\[\begin{split}
\sum_{\substack{
1\le m \le N \\
m\equiv m_0 \;\mod L_S}}
\varphi(m)
=\sum_{\substack{
1\le m \le N \\
m\equiv m_0 \;\mod L_S}}
m\sum_{d|m} \frac {\mu(d)}{d}
=\sum_{\substack{
1\le d\le N\\
d\in \N_S}} \mu(d)
\sum_{\substack{
1\le d' \le N/d\\
d'\in \N_S\\
dd'\equiv m_0\; \mod L_S}} d'.
\end{split}\]

Denote by $m_d$ the unique integer in $\{1, \ldots, L_S -1\}$ for which $d m_d\equiv m_0 \; \mod L_S$.
Let $d'=m_d+L_S(k'-1)$. Since $1\le d'=m_d+L_S(k'-1)\le N/d$, the range of $k'$ is $1\le k' \le N/(d L_S)- m_d/L_S +1$, so that 
\[\begin{split}
\sum_{\substack{
1\le d'\le N/d\\
d'\equiv m_d\;\mod {L_S}}} d'
&=\sum_{k'=1}^{\lfloor \frac{N}{d L_S} - \frac {m_d} {L_S} +1 \rfloor}
(m_d+{L_S} (k'-1))\\
&=m_d\left\lfloor \frac{N }{d L_S} - \frac {m_0(d)}{L_S} +1 \right\rfloor
+\frac{L_S}{ 2} \left(\left\lfloor \frac N {d L_S} - \frac {m_d} {L_S} +1 \right\rfloor^2+\left\lfloor \frac N {d L_S} - \frac {m_d}{ L_S} +1 \right\rfloor\right).
\end{split}\]

Since we want to compute the summation of $\varphi(m)$ up to the error bound $O_{L_S}(N\log N)$, for computational simplicity, we replace $\lfloor N/(d L_S) - m_d/L_S+1 \rfloor$ with $N/(d L_S)$. Equivalently, one can proceed by taking an upper bound $N/(d L_S)+1$ and a lower bound $N/(d L_S)$ and reach the same conclusion.

Hence now our claim is that
\[
\sum_{\substack{
1\le d\le N \\
d\in \N_S}} \mu(d) \left(m_d\frac {N} {d L_S} + \frac {L_S} 2 \left( \frac {N} {d L_S}\right)^2 + \frac {L_S} 2 \left( \frac N {d L_S}\right)\right)= \frac 1 {L_S\zeta_S(2)} \frac {N^2} 2 +O_{L_S}(N\log N)
\]
and we have 3 remaining estimates to conclude the proof.
First,
\[\begin{split}
\Bigg|\sum_{\substack{
1\le d \le N\\
d\in \N_S}}
\mu(d)\cdot m_d\frac N {d L_S}\Bigg|
\le \sum_{1\le d \le N} L_S \frac N {d L_S}
= N\sum_{1\le d \le N} \frac 1 d = O(N\log N).
\end{split}\]
Next, using \cref{eq:summu}
\[\begin{split}
\sum_{\substack{
1\le d \le N\\
d\in \N_S}}
\mu(d)\cdot\frac{ L_S}{ 2} \left(\frac N {d L_S}\right)^2
&=\frac {N^2} {2L_S} \sum_{\substack{
1\le d \le N\\
d\in \N_S}} \frac {\mu(d)} {d^2}
=\frac {N^2} {2L_S} \sum_{
d\in \N_S} \frac {\mu(d)} {d^2}
-\frac {N^2} {2L_S} \sum_{\substack{
d\in \N_S\\
d>N}} \frac {\mu(d)} {d^2}\\
&=\frac {N^2} {2L_S} \frac 1 {\zeta_S(2)}
+O_{L_S}\left(N^2\sum_{d=N+1}^\infty \frac 1 {d^2}\right)
=\frac {N^2} {2L_S} \frac 1 {\zeta_S(2)}
+O_{L_S}\left(N\right).
\end{split}\]
Finally,
\[
\left|\sum_{\substack{
1\le d \le N\\
d\in \N_S}}
\mu(d)\cdot\frac{ L_S} 2\cdot \frac N {d L_S}\right|
\le \frac N 2 \sum_{1\le d \le N} \frac 1 d
=O(N\log N). \]

Therefore the lemma follows.
\end{proof}

Before stating a corollary, let us recall Abel's summation formula (\cite[Theorem 0.3]{TenenbaumIntro}). Let $(a_n)_{n=0}^\infty$ be a sequence of complex numbers and let $A(t):=\sum_{0\le n\le t} a_n$, where $t\in \R$. For $N_1<N_2\in \R$ and $\phi\in C^1([N_1,N_2])$, 
\[
\sum_{N_1<n\le N_2} a_n\phi(n)=A(N_2)\phi(N_2)-A(N_1)\phi(N_1)-\int_{N_1}^{N_2} A(u)\phi'(u)du.
\]
\begin{coro}\label{application of Abel summation formula} 
{There is $r_0>0$ such that  
\begin{enumerate}
\item $\Phi_S(\sx)$ is uniformly bounded on $\sx\in \Q_S$ for which $\rd(\sx)\le r_0$;
\item For almost all $\sx\in \Q_S$ with $\rd(\sx) > r_0$, it holds that
\[\Phi_S(\sx)=\frac 1 {L_S \zeta_S(2)} + O_{L_S}\left(\rd(\sx)^{-1} \log \rd(\sx)\right).\]
\end{enumerate}
}
\end{coro}
\begin{proof}
Let $\sx\in \prod_{p\in S} (\Q_p-\{0\})$ be an $S$-arithmetic number such that $(\rd(\sx)-m_0)/{L_S}\notin \Z$, where $m_0=m_0(\sx)$ is defined as in \cref{eq:m_0}. The set of such $\sx\in \Q_S$ has full measure since it contains $(\R-\Q)\times \prod_{p\in S_f} {(\Q_p-\{0\})}$.
When $\sd(\sx)\leq r_0$, notice that ${\Phi_S(\sx)} \leq r_0 \sum_{m=1}^\infty \frac{\varphi(m)}{m^3} <\infty$ since the Dirichlet series $\sum \varphi(m)/m^z$ converges for $\Re(z)>2$.

{Now, we may assume that $\rd(\sx)>L_S$.
It follows that}
\[
\frac 1 {\rd(\sx)}\Phi_S(\sx)=\sum_{\substack{
m\ge \rd(\sx)\\
m\equiv m_0\;\mod L_S }} \frac {\varphi(m)} {m^3}
=\sum_{n\ge \frac {\rd(\sx)-m_0} {L_S}}
\frac {\varphi(m_0+n L_S)} {(m_0+nL_S)^3}.
\]

Put
\[a_n=\varphi(m_0+nL_S)
\quad\text{and}\quad
\phi(n)=(m_0+nL_S)^{-3}.
\]
We set $N_1:=(\rd(\sx)-m_0)/L_S$ and we will let $N_2\rightarrow \infty$. Abel's summation formula gives
\begin{align}
\sum_{n=N_1}^{N_2} \frac {\varphi(m_0+nL_S)} {(m_0+nL_S)^3}
&=\frac 1 {(m_0+N_2L_S)^3}\sum_{0\le n\le N_2} \varphi(m_0+nL_S) \label{eq:abel1}\\
&-\frac 1 {(m_0+N_1L_S)^3}\sum_{0\le n\le N_1} \varphi(m_0+nL_S)\label{eq:abel2}\\
&+3L_S\int_{N_1}^{N_2} \frac 1 {(m_0+uL_S)^4}\sum_{n=1}^u \varphi(m_0+nL_S) du.\label{eq:abel3}
\end{align}
We see immediately that the right hand side of \eqref{eq:abel1} disappears as $N_2\to\infty$. Using Lemma~\ref{Euler ftn-sum formula},
\[\begin{split}
\eqref{eq:abel2}&=-\frac 1 {(m_0+N_1L_S)^3}\sum_{0\le n\le N_1} \varphi(m_0+nL_S)\\
&=-\frac 1 {(m_0+N_1L_S)^3}\left(\frac {(m_0+N_1L_S)^2}
{2L_S\zeta_S(2)} + O_{L_S}\left((m_0+N_1L_S)\log(m_0+N_1L_S)\right)\right)\\
&=-\frac 1 {\rd(\sx)}\cdot \frac 1 {2L_S \zeta_S(2)}+ O_{L_S}\left(\rd(\sx)^{-2}\log \rd(\sx)\right).
\end{split}\]
We now consider 
\[\begin{split}
\eqref{eq:abel3}&=3L_S\int_{N_1}^{N_2} \frac 1 {(m_0+uL_S)^4}\sum_{n=1}^u \varphi(m_0+nL_S) du\\
&=3L_S\int_{N_1}^{N_2}\frac 1 {(m_0+uL_S)^4}
\left(\frac {(m_0+uL_S)^2} {2L_S\zeta_S(2)} + O_{L_S}\left((m_0+uL_S)\log(m_0+uL_S)\right)\right) du\\
&=\frac 3 {2\zeta_S(2)}\int_{N_1}^{N_2}\frac {du} {(m_0+uL_S)^2}
+3L\int_{N_1}^{N_2} O_L\left(\frac {\log(m_0+uL_S)} {(m_0+uL_S)^3}\right) du\\
&=\frac 3 {2L_S\zeta_S(2)}\cdot\frac 1 {\rd(\sx)}+O_{L_S}\left(\rd(\sx)^{-2}\log \rd(\sx)\right).
\end{split}\]
Multiplying by $\rd(\sx)$, we obtain the formula.
\end{proof}

\begin{proof}[Proof of Proposition~\ref{integral over the cone} (2): integrability]
As mentioned before, it suffices to show that the integral in \eqref{eq 2: integral over the cone (2)} is finite. 
Without loss of generality, let us assume that $F$ is a non-negative, bounded and compactly supported function. 
Recall that the map $$(\bsx, \bsy)\in \Q_S^2\times \Q_S^2\mapsto \det(\bsx, \bsy)=(\det(\bx_p, \by_p))_{p\in S}\in \Q_S \mapsto \Phi_S(\rd(\det(\bsx, \bsy)))$$
is uniformly bounded when $\rd(\det(\bsx, \bsy)))\le r_0$, where $r_0>0$ is given as in Corollary~\ref{application of Abel summation formula}. 
Together with the fact that $F$ is compactly supported, integrability is determined by the second part of   Corollary~\ref{application of Abel summation formula}. That is, there is some $C>0$ so that
\begin{align}
\frac 1 {\zeta_S(2)} \int_{(\Q_S^2)^2}
\Phi_S\left(\det(\bsx, \bsy)\right)& F(\bsx, \bsy) d\bsx d\bsy 
{\leq} C+ \frac 1 {L_S \zeta_S(2)^2} \int_{(\Q_S^2)^2} F(\bsx, \bsy) d\bsx d\bsy \nonumber\\ 
&\hspace{-0.8in}+ \frac 1 {\zeta_S(2)} \int_{\{(\bsx,\bsy)\in (\Q_S^2)^2: \rd(\det(\bsx, \bsy)) > r_0\}} F(\bsx, \bsy)O_{L_S}\left( {\rd(\det(\bsx, \bsy))^{-1}}\log \rd(\det(\bsx,\bsy))\right) d\bsx d\bsy. \label{eq:cone2}
\end{align}

Since the first integral of (R.H.S) is finite, let us focus on the second term 
\eqref{eq:cone2}. By the change of variables ${(\bsx, \bsy)}=\sg \left(\begin{array}{cc}
1 & 0 \\
0 & \sx\end{array}\right)$, 
\[\begin{split}
\eqref{eq:cone2}\ll_{L_S}\int_{\{\sx\in \Q_S: \rd(\sx)>r_0\}} \int_{\sg\in \SG_2}
F\left(\sg\left(\begin{array}{cc}
1 & 0 \\
0 & \sx \end{array}\right)\right) {\log \rd(\sx)} d\eta_2(\sg) d\sx.
\end{split}\]
Since $F$ is compactly supported, $\{\sx=\det(\bsx, \bsy)\in \Q_S: F(\bsx, \bsy)\neq 0\}\subseteq (-b_\infty, b_\infty)\times \prod_{p\in S_f} p^{-b_p}\Z_p$ for some $b_\infty>0$ and $b_p\in \N$ ($p\in S_f$). Then
{
\[\begin{split}
\eqref{eq:cone2}&\ll_{L_S,F} \int_{\{\sx\in (-b_\infty, b_\infty)\times \prod_{p\in S_f} p^{-b_p}\Z_p: \rd(\sx)>r_0\}} \log \left(|x_\infty|_\infty \prod_{p\in S_f} |x_p|_p\right) \prod_{p\in S_f} d x_p \cdot d x_\infty\\
&=2\sum_{(k_p)_{p\in S_f}} \prod_{p\in S_f} p^{k_p}\left(1-\frac 1 p\right)\int_{r_0\prod_{p\in S_f} p^{-k_p}}^{b_\infty} \log\left(|x_\infty|_\infty \prod_{p\in S_f} p^{k_p} \right) dx_\infty\\
&=2\sum_{(k_p)_{p\in S_f}} \prod_{p\in S_f} \left(1 -\frac 1 p \right) \int_{r_0}^{b_\infty\prod_{p\in S_f} p^{k_p}} \log |x'_\infty|_\infty d x'_\infty,
\end{split}\]
where we change the variables $x'_\infty$ to $x'_\infty=x_\infty \prod_{p\in S_f} p^{k_p}$ in each summand.
Moreover, the range of $(k_p)_{p\in S_f}$ for the summations above is
\[
(k_p)_{p\in S_f}\in \Z^s :\quad k_p \le b_p \;\text{for each }p\in S_f
\quad\text{and}\quad \prod_{p\in S_f} p^{k_p} > \frac {r_0}{b_\infty}.
\]
Since the number of such $(k_p)_{p\in S_f}$ is finite, one can conclude that \eqref{eq:cone2} is finite.
}
\end{proof}

\section{Applications}\label{sec:applicationproofs}
 We present the proof of the three applications: Error terms in \cref{sec:errortermsproofs}, Khinthine--Groshev Theorems in \cref{sec:KGproofs}, and Logarithm Laws in \cref{sec:loglawsproofs}.

\subsection{Error Terms}\label{sec:errortermsproofs}
This section concludes with the main goal of proving \cref{Schmidt Main Theorem}. Before arriving at this conclusion, we first need some key measure estimates which are given in Proposition~\ref{var upper bound}. The proof of \cref{var upper bound} for $d\geq 3$ is a direct consequence of  \cref{prop:primmean}, \cref{thm:primrog}, and \cref{lem:boundedunits}. However when $d=2$, the proof of \cref{var upper bound} utilizes \cref{integral over the cone} and a technical lemma giving variance bounds (\cref{Schmidt Lemme 3}).


\begin{prop}\label{var upper bound} 
{
Let $A=\prod_{p\in S} A_p$ and $B=\prod_{p\in S} B_p$ be Borel sets with positive volume such that $A_p\subseteq B_p \subseteq \Q_p^d$ for each $p\in S$.
Let $\one_A$ and $\one_B$ be {the indicator functions} of $A$ and $B$, respectively.
}

\begin{enumerate}
\item For $d\ge 3$, there is $C_d>0$ such that 
\[
\int_{\SG_d/\Gamma_d}
\left(\widehat{\one_A}(\sg \Z_S^d)- \frac 1 {\zeta_S(d)}\vol_S(A) \right)^2 d\mu_d(\sg)
\le C_d \vol_S(A).
\] 
\item {$d\ge 3$. Recall that $\# S=s+1$. It follows that
\[
\int_{\SG_d/\Gamma_d}
\left(\widehat{\mathbf 1}_{B-A}(\sg\Z_S^d) - \frac 1 {\zeta_S(d)}\vol_S(B-A)\right)^2 d\mu_d(\sg)
\le (s+1) \:C_d \vol_S(B-A).
\]
}
\item Let $d=2$. There is a constant $r_0>1$, depending only on $S$, such that there exists $\widetilde{C}>0$ so that for all $A$ with $\frac {\vol_S(A)}{(\log\vol_S(A))^{1+s}}>r_0$,
\[\begin{split}
&\int_{C_S}
\left(\sd(\sv)\widehat{\one_A}(\sv^{1/2}\sg \Z_S^2)- \frac 1 {\zeta_S(2)}\vol_S(A) \right)^2 d\mu_2(\sg) d\sv\\
&\hspace{0.5in}\leq \widetilde{C} \vol_S(A) \left( (\log\vol_S(A))^{2+s} +\left[\sum_{p\in S}\prod_{p'\in S- \{p\}} \vol_{p'}(A_{p'})\right] \right).
\end{split}\]

\end{enumerate}
\end{prop}

The proof requires the following lemma for $d=2$.
\begin{lem}\label{Schmidt Lemme 3} For $p\in S$ let $A_p\subseteq \Q_p^2$ be a Borel set with $\vol_p(A_p)>0$ and let $\one_{A_p}$ be an indicator function of $A_p$. We have the following.
\begin{enumerate}
\item \cite[Lemma 5]{Schmidt60} For $t>0$,
\[
\int_{\{(\bx_\infty, \by_\infty)\in (\R^2)^2: |\det(\bx_\infty, \by_\infty)|_\infty \le t\}} 
\one_{A_\infty} (\bx_\infty) \one_{A_\infty} (\by_\infty) d\bx_\infty d\by_\infty \le 8t\vol_\infty(A_\infty).
\]

\item  Let $\chi_\infty$ be a non-negative, non-increasing function  on $[r_0,r_1]$, where $0\leq r_0 < r_1 \leq \infty$ for which $\int_{r_0}^{r_1} \chi_\infty(t)dt<\infty$. Then
\[
\int_{(\R^2)^2} \chi_\infty(|\det(\bx_\infty,\by_\infty)|_\infty)
\one_{A_\infty} (\bx_\infty) \one_{A_\infty} (\by_\infty)
d\bx_\infty d\by_\infty
\le 8 \vol_\infty(A_\infty)\left[r_0 \chi_\infty(r_0) + \int_{r_0}^{r_1} \chi_\infty(t) dt\right].
\]
\item For $p<\infty$ and any $t\in \Z$, 
\[
\int_{\{(\bx_p, \by_p)\in (\Q_p^2)^2: |\det(\bx_p, \by_p)|_p = p^t\}} 
\one_{A_p} (\bx_p) \one_{A_p} (\by_p) d\bx_p d\by_p \le p^t\left(1-\frac 1 {p^2}\right)\vol_p(A_p).
\]
\item Let $\one_A$ be an indicator function of a Borel set $A=\prod_{p\in S} A_p\subseteq \Q_S^2$ with $\vol_S(A)>0$. Let $\chi$ be a non-negative, non-increasing function on $[1,r]$ for some fixed $r >1$. Then
\begin{align}
&\int_{\{(\bsx, \bsy)\in(\Q_S^2)^2: |\det(\bx_p, \by_p)|_p > 1, \;\forall p\in S\}} 
\one_A(\bsx)\one_A(\bsy) \chi(\sd(\det(\bsx, \bsy))) d\bsx d\bsy\label{eq:Slemma4}\\
&\hspace{1.5in}\le  
8 \vol_S(A) \left[\prod_{p\in S_f}\left( (\log_p r )\left(1-\frac 1 {p^2}\right) \right)\right]\cdot  \left[\chi(1)  +\int_1^r \chi(t) dt\right].\nonumber
\end{align}
\end{enumerate}
\end{lem}
\begin{proof}[Proof of \cref{Schmidt Lemme 3}]
In part (2), this is an exercise in integration by parts, for which the case when $r_0=0$ and $r_1 = \infty$ is proved in \cite[Theorem 3]{Schmidt60}. 

For part (3) denote by $\bx_p={\tp{(x_1,x_2)}}$ and $\by_p={\tp{(y_1,y_2)}}$.
Let us partition the set $$\{(\bx_p, \by_p)\in (\Q_p^2)^2 : |\det(\bx_p, \by_p)|_p=p^t\}$$ into two subsets: $(\bx_p, \by_p)$ with $|x_2|_p \le |y_2|_p$, and with $\frac{|x_2|_p} p \ge |y_2|_p$, respectively. In the first case, for each fixed $\by_p$, where we may assume that $y_2\neq 0$, the volume of the set of possible $\bx_p$ is independent of $\by_p$ since the fibered sets only differ by translation:
\[\begin{split}
&\vol_p\left(\left\{\bx_p\in \Q_p^2: \begin{array}{c}
|x_1y_2-x_2y_1|_p =p^t \\
|x_2|_p \le |y_2|_p \end{array}\right\}\right)\\
&=\vol_p\left(\left\{\bx_p\in \Q_p^2: x_2\in y_2\Z_p \text{ and for each } x_2, \text{ we have } x_1\in \left(x_2\frac {y_1}{y_2}+\frac{p^{-t}}{y_2}(\Z_p-p\Z_p)\right)\right\}\right)\\
&=|y_2^{-1}|_p \:p^t\left(1 - \frac 1 p\right) |y_2|_p
=p^t \left(1- \frac 1 p\right).
\end{split}\]
Thus we have
\[\begin{split}
&\int_{\{(\bx_p, \by_p)\in (\Q_p^2)^2: |\det(\bx_p, \by_p)|_p = p^t,\; |x_2|_p \le |y_2|_p\}} 
\one_{A_p} (\bx_p) \one_{A_p} (\by_p) d\bx_p d\by_p\\
&\hspace{0.5in}\le
\int_{\by_p\in \Q_p^2} \vol_p\left(\left\{\bx_p\in \Q_p^2: \begin{array}{c}
|x_1y_2-x_2y_1|_p =p^t \\
|x_2|_p \le |y_2|_p \end{array}\right\}\right)\cdot \one_{A_p}(\by_p)  d\by_p\\
&\hspace{0.5in}=p^t\left(1-\frac 1 p\right)\vol_p(A_p).
\end{split}\]

Similarly, excluding a set of measure 0, we suppose each $\bx_p$ has $x_2\neq 0$ to obtain
\[
\int_{\{(\bx_p, \by_p)\in (\Q_p^2)^2: |\det(\bx_p, \by_p)|_p = p^t,\; \frac{|x_2|_p}{p} \ge |y_2|_p\}} 
\one_{A_p} (\bx_p) \one_{A_p} (\by_p) d\bx_p d\by_p
\le \frac{p^t}{p} \left(1-\frac 1 p \right) \vol_p(A_p).
\]
The result of (3) follows from combining the above two inequalities.

For part (4), since $\chi$ is defined on $[1,r]$, we partition as follows:
\begin{equation*}
\eqref{eq:Slemma4}=\sum_{t_1=1}^{\lfloor \log_{p_1} r \rfloor}\cdots \sum_{t_s=1}^{\lfloor \log_{p_s} r \rfloor}
\int_{\left\{(\bsx, \bsy)\in (\Q_S^2)^2:\scriptsize \begin{array}{c}
|\det(\bx_{p_j}, \by_{p_j})|_{p_j}=p_j^{t_j},{ \;\forall} p_j\in S_f\\
\frac 1 {p_1^{t_1}\cdots p_s^{t_s}}<|\det(\bx_\infty, \by_\infty)|_\infty \le \frac {r} {p_1^{t_1}\cdots p_s^{t_s}}\end{array}\right\}}
\one_A(\bsx)\one_A(\bsy) \chi(\sd(\det(\bsx, \bsy))) d\bsx d\bsy.
\end{equation*}
Disintegrating place by place and applying (3), the above expression gives
\[\begin{split}
\eqref{eq:Slemma4}&\leq\sum_{t_1=1}^{\lfloor \log_{p_1} r \rfloor}\cdots \sum_{t_s=1}^{\lfloor \log_{p_s} r \rfloor}
\left[\prod_{j=1}^s p_j^{t_j}\left(1- \frac 1 {p_j^2}\right) \vol_{p_j}(A_{p_j})\right]\\
&\hspace{0.5in}
\int_{\left\{
\frac 1 {p_1^{t_1}\cdots p_s^{t_s}}< |\det(\bx_\infty, \by_\infty)|_\infty \le \frac r {p_1^{t_1}\cdots p_s^{t_s}}\right\}}
\one_{A_\infty} (\bx_\infty) \one_{A_\infty} (\by_\infty)
\chi\left(p_1^{t_1}\cdots p_s^{t_s}\cdot|\det(\bx_\infty, \by_\infty)|_\infty\right)
d\bx_\infty d\by_\infty.
\end{split}\]

Applying (2) with $\chi_\infty(t)=\chi(p_1^{t_1}\cdots p_s^{t_s} \cdot t)$ for each $(t_1, \ldots, t_s)$ and each $t\in \Big[\frac {1} {p_1^{t_1}\cdots p_s^{t_s}},\frac {r} {p_1^{t_1}\cdots p_s^{t_s}}\Big]$ and zero elsewhere. With a change of variables, the above gives 
\[\begin{split}
\eqref{eq:Slemma4}&\le 8\vol_\infty(A_\infty)\sum_{t_1=1}^{\lfloor \log_{p_1} r \rfloor}\cdots \sum_{t_s=1}^{\lfloor \log_{p_s} r \rfloor}\\
&\hspace{1.4in}
\left[\prod_{j=1}^s p_j^{t_j}\left(1- \frac 1 {p_j^2}\right) \vol_{p_j}(A_{p_j})\right]
\left[\frac{\chi(1)}{p_1^{t_1}\cdots p_s^{t_s}}  +  \int_{\frac 1 {p_1^{t_1}\cdots p_s^{t_s}}}^{\frac r {p_1^{t_1}\cdots p_s^{t_s}}} \chi(p_1^{t_1}\cdots p_s^{t_s} t) dt\right]\\
&\le 8 \vol_S(A) \left[\prod_{p\in S_f} \log_p(r)\left(1- \frac 1 {p^2}\right) \right]\cdot \left[\chi(1)  +\int_1^r \chi(t) dt\right]\end{split}\]
which completes the lemma.
\end{proof}

\begin{proof}[Proof of \cref{var upper bound}]
The case when $d\ge 3$ is a direct consequence of combining \cref{prop:primmean}, \cref{thm:primrog}, and the proof of \cref{lem:boundedunits}. 

For (2), we construct $s+2$ sets $A_0,\ldots, A_{s+1}$ by accumulatively changing one place in the product between $A$ and $B$. That is, we set $A_0=A$ and $A_1=B_\infty\times A_{p_1}\times \cdots \times A_{p_s}$. Then for $j=2,\ldots, s$ set
$$A_j = B_\infty\times\cdots \times B_{p_{j-1}} \times A_{p_j} \times \cdots \times A_{p_s},$$
and let $A_{s+1} = B$. 
Notice that each $A_{j+1}-A_j$, where $j=0, \ldots, s$ is the product of Borel sets in $\Q_p^d$ for $p\in S$ and 
\[
B-A={  \bigcup_{j=0}^s} (A_{j+1}-A_j).
\] 
The lemma follows from (1), using the Cauchy–Schwarz inequality and the fact that for each $j=0, \ldots, s$, we have $\vol_S(A_{j+1}-A_j)=\vol_S(A_{j+1}) - \vol_S(A_j)$.

Now let us concentrate on the case when $d=2$. Recall we defined $\mu_{C_S} = \mu_2\times \vol_S$. By \cref{integral over the cone} and the fact that $\mu_{C_S}(C_S) = 1/L_S$,
\[\begin{split}
&\int_{C_S}
\left(\sd(\sv)\widehat{\one_A}(\sv^{1/2}\sg \Z_S^2)- \frac 1 {\zeta_S(2)}\vol_S(A) \right)^2 d\mu_2(\sg) d\sv \\
&=\frac 1 {\zeta_S(2)} \int_{(\Q_S^2)^2}
\left(\Phi_S(\det(\bsx, \bsy))-\frac 1 {L_S\zeta_S(2)}\right)
\one_A(\bsx)\one_A(\bsy) d\bsx d\bsy
+\frac{1}{2L_S\zeta_S(2)}\sum_{k\in \Z_S^\times} \int_{\Q_S^2} \one_A(\bsx)\one_A(k\bsx) d\bsx.
\end{split}\]
From the proof of \cref{lem:boundedunits} there is a constant $c_1>0$, depending only on $d=2$ and $S$, such that
\begin{equation}\label{eq:Varineq:eq1}
\frac{1}{2 L_S \zeta_S(2)}\sum_{k\in \Z_S^\times} \int_{\Q_S^2} \one_A(\bsx)\one_A(k\bsx) d\bsx \le c_1\vol_S(A).
\end{equation}

{For the other integral, set $r_0 >1$ as in \cref{application of Abel summation formula} and define a function $\chi$ on $r\in (0,\infty)$ as an upper bound of $|\Phi_S(r) - 1/(L_S\zeta_S(2))|$ by}
	$$\chi(r) =\begin{cases}c_2, &\text{if } 0 <r\leq r_0;\\  
	c_3 r^{-1}\log r, &\text{if } r_0 < r \leq \frac{\vol_S(A)}{(\log \vol_S(A))^{1+s}};\\
	c_4(\vol_S(A))^{-1}(\log\vol_S(A))^{2+s}, &\text{if } r > \frac{\vol_S(A)}{(\log \vol_S(A))^{1+s}},
		\end{cases}$$
  where one can choose that $c_2, c_3, c_4>0$, {depending only on $r_0$}, so that $\chi$ is non-increasing.
		Notice that the upper bound $\chi$ is chosen to optimize the exponents when estimating error terms, and it reduces to the case of \cite[\S 9]{Schmidt60} when $s=0$.
		
Thus we now have the following upper bound partitioned into two types of integrals:
\begin{align}
&\frac 1 {\zeta_S(2)} \int_{(\Q_S^2)^2}
\left(\Phi_S(\det(\bsx, \bsy))-\frac 1 {L_S\zeta_S(2)}\right)
\one_A(\bsx)\one_A(\bsy) d\bsx d\bsy\nonumber\\
& \leq \int_{(\Q_S^2)^2} \chi(\sd(\det(\bsx, \bsy))) \one_A(\bsx)\one_A(\bsy) d\bsx d\bsy \nonumber\\
&\leq \norm{\chi}_{\sup} \sum_{p\in S} \int_{\left\{(\bsx,\bsy)\in (\Q_S^2)^2: \abs{\det(\bx_p, \by_p)}_p \leq r_0\right\}} \one_A(\bsx)\one_A(\bsy) d\bsx d\bsy \label{eq:1leq}\\
&+ \int_{\left\{(\bsx,\bsy)\in (\Q_S^2)^2: \abs{\det(\bx_p, \by_p)}_p >r_0,\,\forall \, p\in S\right\}}\chi(\sd(\det(\bsx, \bsy))) \one_A(\bsx)\one_A(\bsy) d\bsx d\bsy.\label{eq:allgeq}
\end{align}
We first consider the case of \eqref{eq:1leq}. Fix $p\in S$. When $p=\infty$ we apply \cref{Schmidt Lemme 3} (1), and when $p<\infty$, we apply \cref{Schmidt Lemme 3} (3) and sum over all $t$ with $-\infty \leq t \leq \lfloor \log_p(r_0)\rfloor$, which gives
\begin{align*}
	\eqref{eq:1leq} &\leq C_1\sum_{p\in S}\vol_p(A_p) \prod_{p'\in S-\{p\}} \vol_{p'} (A_{p'})^2 = C_1 \vol_S(A) \left[\sum_{p\in S}\prod_{p'\in S- \{p\}} \vol_{p'}(A_{p'})\right].
\end{align*}

We now consider the case of \cref{eq:allgeq}. Assuming that $\frac {\vol_S(A)}{(\log\vol_S(A))^{1+s}}>r_0>1$, we can bound all determinants below by $1$. We apply \cref{Schmidt Lemme 3} (4) on $\left[1,\frac {\vol_S(A)}{(\log\vol_S(A))^{1+s}}\right]$ for $\chi$ to obtain, 
\begin{align*}
\eqref{eq:allgeq} &\leq 8 \vol_S(A) \left[\prod_{p\in S_f} \left(\log_p\frac{\vol_S(A)}{(\log \vol_S(A))^{1+s}}\right)\left(1- \frac{1}{p^2}\right)\right] \left[{ c_2} + \int_1^{\frac{\vol_S(A)}{(\log \vol_S(A))^{1+s}}} \chi(t)\,dt\right]\\
&\hspace{.25in}+c_4 \vol_S(A)(\log\vol_S(A))^{2+s}\\
&\leq 8 \vol_S(A) \left[\prod_{p\in S_f} \left(\log_p(\vol_S(A))\right)\left(1- \frac{1}{p^2}\right)\right]\left[{ c_2r_0} + \frac{c_3}{2} \left(\log \frac{\vol_S(A)}{(\log \vol_S(A))^{1+s}}\right)^2\right]\\
&\hspace{.25in}+c_4 \vol_S(A)(\log\vol_S(A))^{2+s}\\
&\leq C_2 \vol_S(A)(\log\vol_S(A))^{2+s},
\end{align*}
where in the last line we note that $C_2$ depends only on the set $S$ (and the function $\chi$), and we used
$$\prod_{p\in S_f} \log_p(x) = \frac{(\log x)^s}{\prod_{p\in S_f} \log(p)}.$$

 We conclude the case when $d=2$ by setting $\widetilde{C}=\max\{C_1, C_2\}$.
\end{proof}

We now prove \cref{Schmidt Main Theorem}.
\begin{proof}[Proof of Theorem~\ref{Schmidt Main Theorem} (1)]
Let $d\geq 3$.

Denote $K=\min_{p\in S_f} \min\{T_p: \T=(T_p)_{p\in S}\in \mathcal T\}$. Enlarging $\mathcal T$ by adding appropriate Borel sets $A_\T$ if necessary, we may assume that $\T\in \mathcal T$ whenever there is $\T'\in \mathcal T$ for which $\T\succeq \T'$.


Fix $\alpha >1$ and $0<\beta<1$ to be chosen later. We will use the Borel--Cantelli Lemma to show the following is a null set:
	\begin{equation}\label{eq:3BCgoal}
	{ \limsup_{\scriptsize\substack{\prod_{p\in S}T_p  \to \infty\\\T= (T_p)_{p\in S}\\ T_p\geq K,\; \forall p\in S}}} \left\{\sg \in \SG_d/\G_d: \abs{\widehat{\one_{A_\T}}(\sg \Z_S^d) - \frac{1}{\zeta_S(d)} \vol_S(A_\T)} > \vol_S(A_{\T})^\beta  \right\}
	\end{equation}
	To optimize the error term, we will interpolate between sets of volume $k^\alpha$ and sets of volume $(k+1)^\alpha$ for $k\in \N$. Notice that for any $A_1\subseteq A \subseteq A_2 \subseteq \Q_S^d$,
\begin{equation}\label{eq:3interp}
\begin{split}
&\left|\widehat{\one_{A}}(\sg\Z_S^d)-\frac 1 {\zeta_S(d)} \vol_S(A)\right|-\frac 1 {\zeta_S(d)} \vol_S(A_2-A_1)\\
&\hspace{0.2in}\le
\max\left\{
\left|\widehat{\one_{A_1}}(\sg\Z_S^d)-\frac 1 {\zeta_S(d)} \vol_S(A_1)\right|,
\left|\widehat{\one_{A_2}}(\sg\Z_S^d)-\frac 1 {\zeta_S(d)} \vol_S(A_2)\right|
\right\}.
\end{split}\end{equation}

Suppose $\sg$ is an element of the limsup set in \eqref{eq:3BCgoal} and let $\T=(T_p)_{p\in S}$ be large enough (i.e., $\prod_{p\in S}T_p$ is large enough) satisfying the following inequality
\[
\abs{\widehat{\one_{A_\T}}(\sg \Z_S^d) - \frac{1}{\zeta_S(d)} \vol_S(A_\T)} > \vol_S(A_{\T})^\beta.
\]
There is $k\in \N$ for which $k^\alpha \le \vol_S(A_\T) \le (k+1)^\alpha$. Then, there are $\T_1$ and $\T_2$ such that $\T_1 \preceq \T \preceq \T_2$, and $\vol_S(A_{\T_1})=k^\alpha$ and $\vol_S(A_{\T_2})=(k+1)^\alpha$. For example, one can choose $\T_1=(T^{(1)}_p)_{p\in S}$ and $\T_2=(T^{(2)}_p)_{p\in S}$ such that $T^{(1)}_p=T_p=T^{(2)}_p$ for $p\in S_f$ and $T^{(1)}_\infty=\frac {k^\alpha} {\prod_{p\in S_f} T_p}$ and $T^{(2)}_\infty= \frac {(k+1)^\alpha} {\prod_{p\in S_f} T_p}$.

By \eqref{eq:3interp},
\[\begin{split}
\frac{(k+1)^\alpha - k^\alpha}{\zeta_S(d)} &+\max\left\{
\left|\widehat{\one_{A_{\T_1}}}(\sg\Z_S^d)-\frac 1 {\zeta_S(d)} \vol_S(A_{\T_1})\right|,
\left|\widehat{\one_{A_{\T_2}}}(\sg\Z_S^d)-\frac 1 {\zeta_S(d)} \vol_S(A_{\T_2})\right|
\right\}\\
 &\geq \abs{\widehat{\one_{A_\T}}(\sg \Z_S^d) - \frac{1}{\zeta_S(d)} \vol_S(A_\T)} > \vol_S(A_{\T})^\beta {  \ge} k^{\alpha\beta}\end{split}
 \]
 If $A_{\T_1}$ achieves the maximum, then for $k$ large enough (i.e., for $\T$ large enough) and $\alpha-1 < \alpha \beta$, we have
 \begin{equation}\label{eq:Ak}\left|\widehat{\one_{A_{\T_1}}}(\sg\Z_S^d)-\frac 1 {\zeta_S(d)} \vol_S(A_{\T_1})\right| \geq k^{\alpha\beta}\left(1 - \frac{1}{\zeta_S(d)} \frac{(k+1)^\alpha- k^\alpha}{k^{\alpha\beta}} \right) > 0.9 k^{\alpha\beta}.
 \end{equation}
 Otherwise $A_{\T_2}$ achieves the maximum, and again for $k$ large enough and $\alpha-1 < \alpha\beta$, we obtain
	 \begin{equation}\label{eq:Ak1}\left|\widehat{\one_{A_{\T_2}}}(\sg\Z_S^d)-\frac 1 {\zeta_S(d)} \vol_S(A_{\T_2})\right| \geq (k+1)^{\alpha\beta}\left(\frac{k^{\alpha\beta}}{(k+1)^{\alpha\beta}} - \frac{1}{\zeta_S(d)} \frac{(k+1)^\alpha- k^\alpha}{(k+1)^{\alpha\beta}} \right) > 0.9 (k+1)^{\alpha\beta}.\end{equation}

  
	 Thus for $\alpha-1 < \alpha \beta$, the limsup set in \eqref{eq:3BCgoal} is contained in
	 \begin{equation}\label{eq:3BCup}
	 \begin{split} &\limsup_{\scriptsize \substack{
	 k\in \N\\
	 k\to\infty}} \left\{\sg \in \SG_d/\G_d: \abs{\widehat{\one_{A_\T}}(\sg \Z_S^d) - \frac{1}{\zeta_S(d)}k^\alpha} > 0.9 k^{\alpha \beta}\;\text{for some}\; { A_\T\;\text{with}\;\vol_S(A_\T)=k^\alpha}\right\}. \\
	 \end{split}
	 \end{equation}
	
	Applying Chebyshev's inequality along with \cref{var upper bound} to \eqref{eq:3BCup}, and since there are at most $\left(\prod_{p\in S_f} {  \log_p (k/K)^\alpha}\right)$ total number of $A_\T$ in $\mathcal F$ with ${\vol_S(A_\T)}=k^\alpha$ and each $T_p \geq K$,
	\[\begin{split}
 &\mu_d\left(\left\{\sg \in \SG_d/\G_d: \abs{\widehat{\one_{A_\T}}(\sg \Z_S^d) - \frac{1}{\zeta_S(d)}k^\alpha} > 0.9 k^{\alpha \beta}{\;\text{for some}\;A_\T\;\text{with}\; \vol_S(A_\T)=k^\alpha} \right\}\right) \\
&\hspace{1in}\leq C_d'k^{\alpha(1-2\beta)}{\prod_{p\in S_f} \log_p (k/K)^\alpha}.
\end{split}\]
    By the Borel--Cantelli Lemma applied with $\alpha(1-2\beta) < -1$ we have \eqref{eq:3BCgoal} is indeed a null set. We conclude the proof by noting $\alpha = 3$ and $\beta > \frac{2}{3}$ suffices for the desired inequalities.
\end{proof}

\begin{proof}[Proof of \cref{Schmidt Prim+Cong Theorem}]
Let $\{A_{T_\infty}\}_{T_\infty\in \R_{>0}}$, $N=p_1^{k_1}\cdots p_s^{k_s}$, $\bv_0\in P(\Z^d)$ and $\delta>0$ be given. Set $S=\{p_1, \ldots, p_s\}$. Take $\mathcal T=\left\{\T=(T_p)_{p\in S}: T_\infty\in \R_{>0}\;\text{and}\;T_{p_i}=p_i^{-k_i}\;(1\le i\le s)\right\}$. For $\T\in \mathcal T$, define
\[
\SA_{\T}=A_{T_\infty}\times \prod_{i=1}^s (\bv_0+p_i^{k_i}\Z_{p_i}).
\]

Applying \cref{Schmidt Main Theorem} (1) to $\left\{\SA_{\T}\right\}_{\T\in \mathcal T}$, for almost all 
$\sg\in \SL_d(\R)\times \prod_{i=1}^s \mathcal U_{p_i},$
where $\mathcal U_{p_i}$ is an open neighborhood of $\Id$ in $\SL_d(\Q_{p_i})$ such that $$\mathcal U_{p_i}, \;\mathcal U_{p_i}^{-1}\subseteq \Id+p_i^{k_i}\:\Mat_{d}(\Z_p),$$
it holds that 
\[
\#\left(\sg P(\Z_S^d)\cap \SA_{\T}\right)
=\frac 1 {\zeta_S(d)} \vol_S(\SA_\T)+ O_{\sg}\left(\vol_S(\SA_\T)^\delta\right).
\]

Since for each $\sg=(g_p)_{p\in S}\in \SL_d(\R)\times\prod_{i=1}^s \mathcal U_{p_i}$,
\[
(g_\infty, g_{p_1}, \ldots, g_{p_s})P(\Z_S^d) \cap A_{\T}
=(g_\infty, \Id, \ldots, \Id)P(\Z_S^d) \cap A_{\T},
\]
one can deduce that for almost all $g_\infty \in \SL_d(\R)$,
\begin{equation}\label{eq: prim+cong 1}
\#\left((g_\infty, \Id, \ldots, \Id) P(\Z_S^d)\cap \SA_{\T}\right)
=\frac 1 {\zeta_S(d)N^d} \vol_\infty(A_{T_\infty})+ O_{g_\infty,N}\left(\vol_S(\SA_{T_\infty})^\delta\right).
\end{equation}
Note above that the factor of $N^d$ comes from a direct computation of the volume of the $p$-adic sets for $p\in S_f$. We now consider the left hand side of \cref{eq: prim+cong 1}.

Observe that elements of $(g_\infty, \Id, \ldots, \Id)P(\Z_S^d)$ are
\[\left\{
\bv\in P(\Z_S^d): g_\infty\bv\in A_{T_\infty}\;\text{and}\; \bv\in \bv_0+p_i^{k_i}\Z_{p_i}\;(1\le i\le s)\right\}.
\]
Since $\bv\in \Z_{p_i}^d$ for $1\le i\le s$, we have that $\bv \in \Z^d$ (recall that $\Z[1/p]\cap \Z_p=\Z$). Since $\bv\in \Z^d_{p_i}-p_i\Z^d_{p_i}$ (from the choice $\bv_0\in P(\Z^d)$), $\bv$ is not divided by any $p\in S_f$, hence one can conclude that $\bv\in P(\Z^d)$.
It is also easy to check that $\bv\equiv \bv_0\;\mod\; p_i^{k_i}$ for $1\le i\le s$ if and only if $\bv\equiv \bv_0\;\mod\; N=p_1^{k_1}\cdots p_s^{k_s}$. Therefore the left hand side of \cref{eq: prim+cong 1} is equal to the number of $\bv\in P(\Z^d)$ for which $g_\infty \bv\in A_{T_\infty}$ and $\bv\equiv \bv_0\;\mod\; N$, completing the proof.
\end{proof}

\begin{proof}[Proof of \cref{Schmidt Main Theorem} (2)]  Set $d=2$.
From \cref{var upper bound} (3) and by Chebyshev inequality, there is $\widetilde{C}'>0$ such that for all sufficiently large $\T_\ell$, 
\[\begin{split}
&\mu_{C_S} \left(\left\{\sv^{1/2} \sg\in C_S : \begin{array}{c}
\left|\sd(\sv) \widehat{\one}_{A_{\T_\ell}}(\sv^{1/2}\sg\Z_S^2) - \frac 1 {\zeta_S(2)} \vol_S(A_{\T_\ell})\right|\ge\\
\sum_{p\in S} \vol_p((A_\T)_p)^{\delta_1} \prod_{p'\in S-\{p\}} \vol_{p'}((A_{\T_\ell})_{p'}) + \vol_S(A_{\T_\ell})^{\delta_2}
\end{array}
\right\}\right)\\[0.1in]
&\le \frac
{\widetilde{C}\sum_{p\in S} \vol_p((A_{\T_\ell})_p)\prod_{p'\in S-\{p\}} \vol_{p'} ((A_{\T_\ell})_{p'})^2 + {\widetilde{C}'} \vol_S(A_{\T_\ell})^{1+\delta'}}
{\left(\sum_{p\in S} \vol_p((A_{\T_\ell})_p)^{\delta_1}\prod_{p'\in S-\{p\}} \vol_{p'} ((A_{\T_\ell})_{p'}) + \vol_S(A_{\T_\ell})^{\delta_2}\right)^2}\\[0.1in]
&\le \frac
{\widetilde{C}\sum_{p\in S} \vol_p((A_{\T_\ell})_p)\prod_{p'\in S-\{p\}} \vol_{p'} ((A_{\T_\ell})_{p'})^2 + {\widetilde{C}'} \vol_S(A_{\T_\ell})^{1+\delta'}}
{\sum_{p\in S} \left(\vol_p((A_{\T_\ell})_p)^{\delta_1}\prod_{p'\in S-\{p\}} \vol_{p'} ((A_{\T_\ell})_{p'})\right)^2 + \left(\vol_S(A_{\T_\ell})^{\delta_2}\right)^2}\\[0.1in]
&\le \widetilde{C}\sum_{p\in S} \vol_p((A_{\T_\ell})_p)^{1-2\delta_1}
+\widetilde{C}'\vol_S(A_{\T_\ell})^{1+\delta'-2\delta_2}.
\end{split}\]
By our assumption for $(\T_\ell)$, $\delta_1$ and $\delta_2>0$,
the result follows from the Borel--Cantelli lemma.
\end{proof}

\begin{proof}[Proof of \cref{thm:starshaped}]
We first consider the case $S=\{\infty\}$. In this case, \cite[Theorem 1]{Schmidt60} states that the error term is given by 
$$O(\vol_S(\T A)^\frac{1}{2} \log(\vol_S(\T A)) \psi(\log(\vol_S(\T A)))$$
for a positive nonincreasing function $\psi$ on $\mathbb{R}_{\geq 0}$ so that $\int_0^\infty \psi^{-1} < \infty$. By considering $\psi$ with $\psi(s) = s^2$ for $s\geq 1$ and $\psi(s) = 1$ for $s<1$, we obtain the formula in \cref{thm:starshaped}. The proof strategy requires reducing the theorem for those $\T A$ with $\vol_S(\T A) \in \mathbb{N}$ and then applying \cite[Lemmas 2 and 3]{Schmidt60}.

For general $S$, we can use the same function $\psi$, and a similar proof by only needing to adapt the two lemmas from \cite{Schmidt60}. Namely, we can first reduce the theorem statement to those $\T A$ with $\vol_S(\T A)\in \vol_S(A)\N$.


For general $S$, one can reduce the theorem for those $\T A$ such that $\vol_S(\T A)\in \vol_S(A)\N$. For each $T\in \N$, set
\[
K_T=\left\{(\T_1, \T_2)\in \Big(\R_{>0}\times \prod_{p\in S_f} p^{(\N \cup \{0\}}\Big)^2: \begin{array}{c}
\T_1 \preceq \T_2;\quad T_p^{(1)}=T_p^{(2)}\;\text{for}\; \forall p\in S_f;\\
0\le \Big(\hspace{-0.03in}\prod\limits_{p\in S} T_p^{(1)}\hspace{-0.03in}\Big)^d\hspace{-0.05in}=u2^t < \Big(\hspace{-0.03in}\prod\limits_{p\in S} T_p^{(2)}\hspace{-0.03in}\Big)^d\hspace{-0.05in}= (u+1)2^t\le 2^T\\
\text{for some non-negative integers $u$ and $t$}\end{array}\hspace{-0.05in}\right\}.
\]
Applying \cref{var upper bound} (2) the analog of \cite[Lemma 2]{Schmidt60} is
\begin{equation}\label{Sch lem 2}\begin{split}
\mathrm{SV}_T(A):=&\sum_{(\T_1, \T_2)\in K_T} \int_{\SG_d/\Gamma_d} 
\left(\widehat \one_{\T_2 A-\T_1 A} - \frac 1 {\zeta_S(d)} \vol_S(\T_2 A -\T_1 A)\right)^2 d\mu_d(\sg)\\
&\hspace{2in}\le (s+1)C_d (T+1)2^T\vol_S(A)\cdot \prod_{p\in S_f} \log_p (2^{T/d}).
\end{split}\end{equation}
Here we note that $\prod_{p\in S_f} \log_p(2^{T/d})$ is the upper bound of the number of $(\T_1, \T_2)\in K_T$ for which $\left(\prod_{p\in S} T_p^{(1)}\right)^d=u2^t$.

For the analog of \cite[Lemma 3]{Schmidt60} we apply \cref{Sch lem 2} to get
\begin{equation}\label{Sch lem 3}\begin{split}
&\mu_S\left(\left\{\sg\Gamma_d\in \SG_d/\Gamma_d: 
\mathrm{SV}_T(A)>(T+1)2^T \Big(\prod_{p\in S_f}\log_p (2^{T/d})\Big) \psi(T\log 2-1)\vol_S(A)\right\}\right)\\
&\hspace{2in}<(s+1)C_d\psi^{-1}(T\log 2-1).
\end{split}\end{equation}

When we follow the identical argument with the proof of \cite[Theorem 1]{Schmidt60} using \cref{Sch lem 3} instead, and $\psi(s)=s^2$. In doing so we verify that the complement of the limsup set of the set given in \eqref{Sch lem 3} over all $T\in \N$ is a {full-measured} set of $\SG_d/\Gamma_d$ satisfying the formula in \cref{thm:starshaped} .
\end{proof}

\subsection{Khintchine--Groshev Analogs}\label{sec:KGproofs}

In the proof of the \cref{1-d primitive Khintchine-Groshev Thm}, we will briefly follow footprints of the idea used in \cite[Section 4]{Han22}, which was introduced in \cite{AGY21} with gentle modification to the case when $m=n=1$.

\begin{proof}[Proof of \cref{1-d primitive Khintchine-Groshev Thm}]
Note that 
\[
\widehat{N}_{\psi,x}(\T)=\# \:\su_{\sx} P(\Z_S^2) \cap E_{\psi}(\T),
\;\text{where}\;\su_{\sx}=\left(\begin{array}{cc}
1 & \sx \\
0 & 1 \end{array}\right)\]
and $V_\psi(\T)=\vol_S(E_\psi(\T))$.

The trick used here is to reduce the ``almost all $\su_\sx$''-statement from the ``almost all $\sg$''-statement of \cref{Schmidt Main Theorem} (2) using an approximating technique.  
For this, let us take a decreasing sequence $(\eps_\ell)_{\ell\in \N}$ converging to $0$ and define
$\psi^{\pm}_{\ell}=(\psi^{\pm}_{p,\ell})_{p\in S}$ by
\[
\psi^{\pm}_{p,\ell}(|\sy|_p)
=\left\{
\begin{array}{cl}
(1+\eps_\ell)^{\pm 1} \psi_{p}\left(\dfrac 1 {(1+\eps_\ell)^{\pm1}} |\sy|_p \right), &\text{for }p=\infty;\\[0.15in]
\psi_p(|\sy|_p), &\text{for }p\in S_f.
\end{array}\right.
\]
Also, we will consider sequences $(\T^{\pm}_\ell)_{\ell\in \N}$ defined by
\[
\T^{\pm}_\ell
=((1+\eps_\ell)^{\pm 1}T^{(\ell)}_\infty, T^{(\ell)}_{p_1}, \ldots, T^{(\ell)}_{p_s}),\; \ell\in \N.
\]
Notice that these sequences $(\T^{\pm}_\ell)$ also satisfies the conditions in \cref{borel-cantelli condition}, and 
it is not hard to show that $\vol_S(E_{\psi^{\pm}_\ell})(\T^{\pm}_\ell)=(1+\eps_\ell)^{\pm 2} \vol_S(E_{\psi}(\T_\ell))$.

Applying \cref{Schmidt Main Theorem} (2) to each $\psi^{\pm}_\ell$, one can deduce that for almost all $\sg\in \Interval_1\times \SL_2(\Q_S)$,
\begin{equation}\label{eq thm 5.1 (2) kg}\begin{split} 
&\rd(\det\sg) \#\left(\sg\Z_S^2 \cap E_{\psi^{\pm}_\ell}(\T^\pm_\ell)\right)
=\frac {(1+\eps_\ell)^{\pm 2}} {\zeta_S(2)} \vol_S(E_{\psi}(\T_\ell))\\
&\hspace{0.2in}+ O\left(\sum_{p\in S} \vol_p\left(E_{\psi_p}(T^{(\ell)}_p)\right)^{\delta_1}\prod_{p'\in S-\{p\}} \vol_{p'}\left( (E_{\psi_{p'}}(T_{p'}^{(\ell)})\right)\right)
+ O\left(\vol_S(E_{\psi}(\T_\ell))^{\delta_2}\right).
\end{split}\end{equation}
Note that $\Interval_1\times \SL_2(\Q_S)$ can be decomposed as
\[
\left\{\left(\begin{array}{cc}
\sv^{1/2} & 0 \\
0 & \sv^{  1/2}
\end{array}\right): \sv\in \Interval_1\right\}\cdot
\left\{\left(\begin{array}{cc}
\sa & 0 \\
\sb & \sa^{-1}
\end{array}\right): \begin{array}{c}
\text{invertible } \sa\in \Q_S,\\ 
\sb\in \Q_S\end{array}\right\}\cdot
\left\{\su_\sx : \sx \in \Q_S \right\}.
\]

Let us denote $\sh(\sv, \sa, \sb)=\left(\begin{array}{cc}
\sv^{1/2} & 0 \\
0 & \sv^{  1/2}
\end{array}\right)\left(\begin{array}{cc}
\sa & 0 \\
\sb & \sa^{-1}
\end{array}\right)$ so that {a generic} element of $\Interval_1\times \SL_2(\Q_S)$ can be expressed as $\sg=\sh(\sv, \sa, \sb)\su_{\sx}$.
For each $\ell\in \N$, set
\[
C_S(\eps_\ell)
:=\left\{{  \sh}(\sv, \sa, \sb) :
\begin{array}{c}
v_\infty \in (\frac {\eps_\ell} {16}, 1];\\
v_p\in 1+L_p\Z_p
\end{array}
,\;
\begin{array}{c}
|a_\infty|_\infty^{\pm1}\le 1+\frac {\eps_\ell} 4;\\
a_p \in \Z_p-p\Z_p
\end{array}
\;\text{and}\;
\begin{array}{c}
|b_\infty|_\infty\le 1+\frac {\eps_\ell} 8;\\
b_p \in \Z_p.
\end{array}
 \right\}
\] 
so that for any element $\sh\in C_S(\eps_\ell)$, we have that $$E_{\psi^-_\ell}(\T^-_\ell) \subseteq \sh E_\psi(\T_\ell) \subseteq E_{\psi^+_\ell}(\T^+_\ell).$$

Since $C_S(\eps_\ell)\{\su_\sx: \sx\in \Q_S\}$ is open in $\Interval_1\times \SL_2(\Q_S)$, one can find a sequence $(\sh_\ell=\sh(\sv_\ell, \sa_\ell, \sb_\ell))_{\ell\in \N}$ such that for each $\ell\in \N$, the asymptotic formula \cref{eq thm 5.1 (2) kg} holds for $\sh_\ell \su_\sx$ for almost all $\sx\in \Q_S$. 
Therefore one can find a full-measure set of $\Q_S$ whose element $\sx$ satisfies \cref{eq thm 5.1 (2) kg} for $\sh_\ell \su_\sx$, $\forall \ell\in \N$.

For such $\sx\in \Q_S$, since $\delta_1, \delta_2 <1$, it follows that
\[\begin{split}
&\lim_{\ell\rightarrow \infty}
\left|
\frac {\widehat{N}_{\psi,\sx}(\T_\ell)}{V_{\psi}(\T_\ell)/\zeta_S(2)}-1
\right|
=\left|\frac{\# \sh_\ell (\su_\sx P(\Z_S^2) \cap E_{\psi}(\T_\ell))}{\vol_S(E_{\psi}(\T_\ell))/\zeta_S(2)}-1\right|\\
&\le \lim_{\ell\rightarrow \infty}
\max
\left\{\left|\frac{\#\sh_\ell \su_\sx P(\Z_S^2) \cap E_{\psi^{+}_\ell} (\T^+_\ell)}{\vol_S(E_{\psi}(\T_\ell))/\zeta_S(2)}-1\right|, 
\left|\frac{\#\sh_\ell \su_\sx P(\Z_S^2) \cap E_{\psi^{-}_\ell} (\T^-_\ell)}{\vol_S(E_{\psi}(\T_\ell))/\zeta_S(2)}-1\right|\right\}\\
&\le \lim_{\ell\rightarrow \infty} 3\eps_\ell 
+O\left(\sum_{p\in S} \vol_p\left(E_{\psi_p}(T^{(\ell)}_p)\right)^{\delta_1-1}\right)
+O\left(\vol_S(E_{\psi}(\T_\ell))^{\delta_2-1}\right)
=0.
\end{split}\]
\end{proof}

\begin{proof}[Proof of \cref{primitive Khintchine-Groshev Thm}]
The two cases are almost identical with those of Theorem 1.3 and Theorem 1.4 in \cite[Section 4]{Han22} with $N=1$ and $d=m+n$, respectively, except we use \cref{Schmidt Main Theorem} (1) instead of \cite[Theorem 4.1]{Han22}.
\end{proof}

\begin{proof}[Proofs of \cref{1-d primitive Khintchine-Groshev Thm real} and \cref{quan K--G thm with prim+cong}]
The results follow when we apply the same argument (with $S=\{\infty\}$) in the proof of \cref{1-d primitive Khintchine-Groshev Thm} replacing the use of \cref{Schmidt Main Theorem} with real versions, and using different target sets $E_\psi(T)$. We also use that there is a two-to-one correspondence between $P(\Z^2)$ and $\Q$. In the case of \cref{1-d primitive Khintchine-Groshev Thm real}, we use \cite[Theorem 2]{Schmidt60} and
\[
E_\psi(T)=\{(x,y)\in \R\times \R: |x|_\infty \le \psi(|y|_\infty)\;\text{and}\; |y|_\infty < T\},\text{ for all } T>0.
\]
In the case of \cref{quan K--G thm with prim+cong}, we use \cref{Schmidt Prim+Cong Theorem} and set 
\[
E_{\psi}(T)=\left\{(\bx, \by)\in \R^m\times \R^n : \|\bx\|_\infty^m \le \psi(\|\by\|_\infty^n)\;\text{and}\; \|\by\|_\infty^n<T\right\},\;\forall T>0.
\]

\end{proof}




\subsection{Logarithm Laws}\label{sec:loglawsproofs}

 The proof proceeds in 3 sections: an analog of the Random Minkowski theorem in \cref{sec:randommink}, then upper bounds in \cref{sec:upperlog}, finishing with lower bounds in \cref{sec:lowerlog}.

\subsubsection{Random Minkowski} \label{sec:randommink}

\begin{proof}[Proof of \cref{prop:random minkowski}]
	When $d\geq 3$, set $g_A: \SG_d/\Gamma_d \to \R$ to be $g_A = 1- \mathbf{1}_{\{\sg\Gamma_d\in \SG_d/\Gamma_d: \sg P(\Z_S^d) \cap A = \emptyset\}}$. Then by \cref{prop:primmean}, the Cauchy--Schwarz inequality and \cref{var upper bound}
	$$\frac{\vol_S(A)^2}{\zeta_S(d)^2} = \left(\int_{\SG_d/\Gamma_d} \widehat{\mathbf{1}_A}\,d\mu_d\right)^2 \leq \norm{\widehat{\mathbf{1}_A}}^{2}_{\SG_d/\Gamma_d,2} \norm{g_A}_{\SG_d/\Gamma_d,1}\leq \norm{g_A}_{\SG_d/\Gamma_d,1} \left(C_d\vol_S(A) + \frac{\vol_S(A)^2}{\zeta_S(d)^2}\right),$$
{where we use the fact that $\|g_A\|_{\SG_d/\Gamma_d,2}^2=\|g_A\|_{\SG_d/\Gamma_d, 1}$ provided that $g_A$ is an indicator function.}
	Thus
	$$ \mu_d(\{\sg\Gamma_d\in \SG_d/\Gamma_d: \sg P(\Z_S^d) \cap A = \emptyset\}) \leq \frac{C_d \vol_S(A)}{\left(C_d\vol_S(A) + \frac{\vol_S(A)^2}{\zeta_S(d)^2}\right)} \leq  \frac{C_d \zeta_S^2(d)}{\vol_S(A)}.$$

For $d=2$, we extend the function $g_A$ to a function on $C_S$, also denoted by $g_A$, by $g_A(\sv^{1/2}g\Gamma_2)=g_A(\sv^{1/2}g\Gamma_2)$ for all $\sv\in I_1$.
By \cref{integral over the cone} and the Cauchy-Schwarz inequality applied to the probability space $(C_S, L_S\mu_{C_S})$,
\[\begin{split}
\frac {\vol_S(A)^2}{\zeta_S(2)^2}
&=\left(L_S\int_{C_S} \rd(\sv)\widehat{\mathbf{1}_A}(\sv^{1/2} \sg \Gamma_2) g_A(\sg\Gamma_2)\,d\mu_2 \,d\sv \right)^2\\
&\le
\left(L_S\int_{C_S} \left(\rd(\sv)\widehat{\mathbf{1}_A}(\sv^{1/2} \sg \Gamma_2)\right)^2\,d\mu_2 \,d\sv \right)
\left(L_S\int_{C_S} g_A(\sg\Gamma_2)^2\,d\mu_2 \,d\sv \right).
\end{split}\]
It follows that
\[
L_S\int_{C_S} g_A(\sg\Gamma_2)^2\,d\mu_2 \,d\sv
=L_S\frac 1 {L_S}\int_{\SG_2/\Gamma_2} g_A(\sg\Gamma_2)^2 d\mu_2=\norm{g_A}_{\SG_2/\Gamma_2, 1}.
\]
Applying \cref{var upper bound},
\[
\frac {\vol_S(A)^2}{\zeta_S(2)^2}
\le 
\left(\frac {\vol_S(A)^2}{\zeta_S(2)^2} + \widetilde{C} L_S \vol_S(A) E(A)\right)
\norm{g_A}_{\SG_2/\Gamma_2, 1},
\]
which leads to
\[
\mu_2(\{\sg\Gamma_2\in \SG_2/\Gamma_2: \sg P(\Z_S^2) \cap A = \emptyset\}) 
\leq
\frac {\widetilde{C} L_S \zeta_S(2)^2 E(A)}{\vol_S(A)}.
\]
\end{proof}

\subsubsection{Upper bounds}\label{sec:upperlog}
We next pursue \cref{upper bound of AM}, an upper bound for \cref{thm:loglawAM}.

\begin{proof}[Proof of \cref{upper bound of AM}]
Take any countable sequence $(\eps_r)\subseteq \R_{>0}$, where $\eps_r\rightarrow 0$ as $r\rightarrow \infty$.
It suffices to show that for each $r$, the set of $\Lambda$ satisfying
\[
\limsup_{|\sx|\rightarrow\infty}
\frac{\log(\alpha_1(\su_\sx \Lambda))}{\log\left(\prod_{p\in S} |x_p|_p\right)} \le \frac{1}{d}+\eps_r
\] 
is a full measure set.

Since the map $\sx\mapsto \log(\alpha_1(\su_\sx \Lambda))$ is upper semicontinuous and the map $\sx\mapsto \log(\prod_{p\in S} |x_p|^{-1}_p)$ is continuous on $\prod_{p\in S} (\Q_p -\{0\})$, and since we want to obtain the supremum limit, it is enough to consider those $\sx$ with $|\sx|_p\ge 1$ for $\forall p\in S$, i.e. of the form
\begin{equation}\label{eq upper AM}
\sx=\frac {m} {p_1^{k_1}\cdots p_s^{k_s}} \in \Z_S
\quad\text{s.t.}\quad
\left\{\begin{array}{l}
k_1, \ldots, k_s \in \Z_{\ge 0};\\
\left|{m}/({p_1^{k_1}\cdots p_s^{k_s}})\right|_\infty \ge 1;\\
m\in \N_S\;\text{or}\; -m \in \N_S.
\end{array}\right.
\end{equation}
Notice that $\prod_{p\in S} |x_p|_p=|m|_\infty$.

We want to obtain the upper bound of 
\[
\mu_d\left(\left\{
\Lambda\in \SG_d/\Gamma_d: 
\log(\alpha_1(\su_\sx \Lambda)) > \left(\frac 1 d + \eps_r\right) \log |m|_\infty.
\right\}\right),
\]
where $\sx=m/(p_1^{k_1}\cdots p_s^{k_s})$.
Note that $\alpha_1(\su_\sx \Lambda)>|m|_\infty^{1/d+\eps_r}$ if there is $\bv\in \su_\sx\Lambda$ for which $\prod_{p\in S} \|\bv\|_p < |m|_\infty^{-(1/d+\eps_r)}$. It is a fact that there is $D=D(d, S)>0$ such that by multiplying an element of $\Z_S^\times$ to $\bv$, one can find $\bw\in \su_\sx \Lambda$ such that
\[
\|\bw\|_p < Dm^{-\frac 1 {s+1} \left(\frac 1 d +\eps_r\right)},\; \forall p\in S,
\] 
where $s+1$ is the cardinality of $S=\{\infty, p_1, \ldots, p_s\}$. This implies that if we set $$B=B\left(0, D|m|_\infty^{-\frac 1 {s+1} \left(\frac 1 d +\eps_r\right)}\right)\times \prod_{p\in S_f} \left\lfloor D|m|_\infty^{-\frac 1 {s+1} \left(\frac 1 d +\eps_r\right)}\right\rfloor_p {\Z_p^d},$$ where $\lfloor t \rfloor_p$ is the largest number in $\{p^z:z\in \Z\}$ less than or equal to $t$, we have
\[
\widetilde{\mathbf{1}}_B(\su_\sx\Lambda)\ge \widehat{\mathbf{1}}_B(\su_\sx\Lambda)\ge 1.
\]
Hence applying the mean value formula in \cref{prop:primmean},
\[
\mu_d\left(\left\{
\Lambda\in \SG_d/\Gamma_d: 
\log(\alpha_1(\su_\sx \Lambda)) > \left(\frac 1 d + \eps_r\right) \log |m|_\infty
\right\}\right) \ll_{d,S} \frac 1 {|m|_\infty^{1+d\eps_r}}.
\]

Summing over $\sx\in \Z_S$ with $|\sx|_p\ge 1$ for $\forall p\in S$, following the notation in \cref{eq upper AM}, we have
\[\begin{split}
&\sum_{\scriptsize \begin{array}{c}
\sx\in \Z_S\\
\text{as in }\cref{eq upper AM}
\end{array}}
\mu_d\left(\left\{
\Lambda\in \SG_d/\Gamma_d: 
\log(\alpha_1(\su_\sx \Lambda)) > \left(\frac 1 d + \eps_r\right) \log |m|_\infty
\right\}\right)\\
&\ll 2 \sum_{k_1\in \Z_{\ge 0}} \cdots \sum_{k_s\in \Z_{\ge 0}} 
\sum_{\scriptsize \begin{array}{c}
m\in \N_S\\
\end{array}}
\frac 1 {m^{1+\eps_r d}}\\
&\le 2 \sum_{m \ge 1} 
\frac 1 {m^{1+\eps_r d}} \#\left\{(k_1, \ldots, k_s)\in \N^s: p_1^{k_1} \cdots p_s^{k_s} \le m \right\}\\
&\le 2 \sum_{m \ge 1} 
\frac 1 {m^{1+\eps_r d}} \prod_{p\in S_f} \log_p m
<\infty.
\end{split}\]
Hence we achieve our claim by the Borel--Cantelli lemma.
\end{proof}

\subsubsection{Lower bounds} \label{sec:lowerlog} 
To proceed with \cref{lower bound of AM}, the lower bound, let us first recall some facts about unipotent one-parameter subgroups in the $S$-arithmetic setting that we need for the proof of \cref{lower bound of AM}.  To do so we just need some observations for unipotent one-parameter subgroups in $\SL_d(\Q_p)$ for a prime $p$, which mostly mimic the real case. 
Recall the matrix exponential map and the matrix logarithmic map
\[
\exp(X)=\sum_{i=0}^\infty \frac {X^i} {i!}
\quad\text{and}\quad
\log(X)=\sum_{i=1}^\infty (-1)^{i+1} \frac {(X-\Id)^i} {i}
\]
on the space of $d\times d$ matrices, as formal power series. Note that the convergence of the exponential and logarithmic maps with respect to $p$-adic numbers behaves differently than the real case, but we avoid this subtlety since we are considering unipotent and nilpotent matrices. 
In particular, one is the inverse of the other.

Let $U_t$ be a unipotent one-parameter subgroup, i.e. a continuous homomorphism from $t\in \Q_p$ to $U_t\in \SL_d(\Q_p)$.
Then since the map $t\mapsto \log U_t$ is a continuous homomorphism in $(\Mat_{d}(\Q_p), +)$ and $\Z$ is dense in $\Z_p$, by evaluating at $t=1/p^k$ for $\forall k\in \N$, $t\in \frac 1 {p^k} \Z$, and then $t\in \frac 1 {p^k} \Z_p$, we can construct the nilpotent element $N:=\log U_1\in \Mat_{d}(\Q_p)$ for which $\log U_t=Nt$. 

In other words, for any unipotent one-parameter subgroup, one can find a nilpotent $N\in \Mat_{d}(\Q_p)$ so that 
\[
U_t= \exp(Nt),\; \forall t\in \Q_p.
\]
Using the Jordan-canonical form for nilpotent $N$, we further obtain that there is $h\in \SL_d(\Q_p)$ so that $h^{-1}U_th$ is a diagonal of block matrices where each block is of the form 
\begin{equation}\label{Jordan block}
\begin{pmatrix}1 & t & t^2/2 &t^3/6 &\cdots &t^k/k!\\
0&1& t & {t^2}/2 &\cdots & t^{(k-1)}/(k-1)!\\
\vdots & 0 & \ddots & \ddots & & \vdots\\
\vdots & \cdots & 0 & \ddots & \ddots& \vdots\\
0& \cdots & \cdots &0 & 1 & t\\
0&\cdots &&\cdots & 0 & 1
\end{pmatrix}, \quad t\in \Q_p.
\end{equation}

Therefore, for any unipotent one-parameter subgroup $\su_\t$, there is $\sh\in \SL_d(\Q_S)$ so that each $p$-adic matrix of $\su'_t:=\sh^{-1}\su_\t \sh$ is a diagonal of Jordan blocks, where the sizes of blocks can be different in each place. By replacing the variable $\Lambda$ by $\Lambda'=\sh^{-1}\Lambda$ and using the fact that $\log\alpha_1(\su'_\t \Lambda')=\log\alpha_1(\sh^{-1}\su_\t \Lambda)$ differs from $\log\alpha_1(\su_\t \Lambda)$ by a uniform bound, it suffices to show that for $\mu_d$-almost every $\Lambda'$
\[
\limsup_{|\sx|\rightarrow \infty}
\frac {\log(\alpha_1(\su'_\sx \Lambda'))} {\log(\prod_{p\in S} |x_p|_p)} \ge \frac 1 d.
\]

Therefore, from now on, we may assume $\su_\t$ consists of matrices of Jordan normal form.

\begin{proof}[Proof of \cref{lower bound of AM}]
Since the key ideas are contained in the case of $d=3$, we first consider $d=3$, and then generalize to the case of $d\ge 3$. From the difference of assumptions for the cases of $\dim\ge 3$ and $\dim=2$ respectively in \cref{cor:limit rand minkowski}, the latter case demands more process, which we address at the end of the proof.

\vspace{0.1in}
\noindent\textbf{The case of $\dim=3$.}\quad
In this case, $\su_\t$ consists of matrices of the form
\begin{equation}\label{3-dim lower bound type}
\left(\begin{array}{ccc}
1 & 0 & 0 \\
0 & 1 & t_p \\
0 & 0 & 1
\end{array}\right)\;
\quad\text{or}\quad
\left(\begin{array}{ccc}
1 & t_{p} & {t_{p}^2}/2 \\
0 & 1 & t_{p} \\
0 & 0 & 1
\end{array}\right),\quad t_{p}\in \Q_{p}.
\end{equation}

It suffices to show that for any $\delta>0$, it follows that for almost every $\Lambda$, there is a sequence $(\t_\eta)_{\eta\in \N}$ of $\Q_S$ so that 
\begin{equation}\label{claim: lower bound}
\log(\alpha_1(\su_{\t_\eta} \Lambda)) \ge \left(\frac 1 d -\delta\right)\log\left(\prod_{p\in S} |\t_\eta|_p\right).
\end{equation}

Fix a constant $\eps>0$ depending on $\delta>0$ to be chosen later. 
Consider a family of sets $A_{\SK}=\prod_{p\in S} A_{\SK}^{(p)}$ for $\SK=(k_\infty, p_1^{k_{p_1}}, \ldots, p_s^{k_{p_s}})\in \N\times \prod_{p\in S_f} p^{3\N}$, where for each $p\in S$,
\begin{enumerate}
\item if the $p$-adic element of $\su_\t$ is the 
matrix on the left in \cref{3-dim lower bound type}, $A_{\SK}^{(p)}$ is the set of $(x_p,y_p,z_p)\in\Q_p^3$ given by either
\begin{equation}\label{eq:A_k}
\begin{cases}
0<|x_\infty|_\infty \le k_\infty^{-1/3},\\
|y_\infty|_\infty \le k_\infty^{2/3}, {\frac{y_\infty}{z_\infty}<0}\\
k_\infty^{-1/3-\eps} < { |z_\infty|_\infty} < k_\infty^{-1/3+\eps}; 
\end{cases}
\text{or if } p<\infty, \quad
\begin{cases}
x_p\in p^{k_p/3}\Z_p-p^{(k_p/3)+1}\Z_p, \\
{  y_p}\in p^{-2k_p/3}\Z_p-p^{-(2k_p/3)+1}\Z_p, \\
z_p\in p^{k_p/3}\Z_p-p^{(k_p/3)+1}\Z_p,
\end{cases}
\end{equation}
\item if the $p$-adic element of $\su_\t$ is the 
matrix on the right in \cref{3-dim lower bound type}, then $A_{\SK}^{(p)}$ is given by the same inequalities as \cref{eq:A_k} except for the first coordinate:
\[
0<\left|x_\infty- \frac {y_\infty^2} {2z_\infty}\right|_\infty \le k_\infty^{-1/3}
\quad{\text{ or if } p<\infty,\quad}
x_p - \frac {y_p^2} {2z_p} \in p^{k_p/3}\Z_p - p^{(k_p/3)+1}\Z_p.
\]
\end{enumerate} 
Note that in both cases, the volumes are
\[
\vol_\infty (A_{\SK}^{(\infty)})=4(k_\infty^{\eps}-k_\infty^{-\eps})
\quad\text{and}\quad
\vol_p (A_{\SK}^{(p)})=(1-1/p)^{3},
\]
so that $\vol_S(A_{\SK})\rightarrow \infty$ when $k_\infty\rightarrow \infty$.

	Select an increasing sequence  $(\SK_\eta)_{\eta\in \N} = (k_{\infty,\eta}, p_1^{k_{p_1,\eta}}, \ldots, p_s^{k_{p_s,\eta}})_{\eta \in \mathbb{N}}$ such that for each $p\in S$, $k_{p,\eta}\rightarrow \infty$ as $\eta\rightarrow \infty$. Since $k_{\infty,\eta} \to \infty$, we can apply \cref{cor:limit rand minkowski} to obtain 
	\[
 \lim_{\eta \to\infty }\mu_3(\Lambda: P(\Lambda) \cap A_{\SK_\eta}=\emptyset)=0.
\]
Thus replacing $\eta$ with a subsequence if necessary we can further assume
\[
\sum_{\eta=1}^\infty \mu_3(\Lambda : P(\Lambda) \cap A_{\SK_\eta}=\emptyset)<\infty.
\]



By the Borel--Cantelli lemma, for almost every $\Lambda$, there exists $\eta_0 = \eta_0(\Lambda)$ so that for all $\eta\geq \eta_0$, there is some $\bsv_\eta = {\tp{(\sx_\eta,\sy_\eta,\sz_\eta)}} \in P(\Lambda) \cap A_{\SK_\eta}$. 

Now, let us fix such a lattice $\Lambda$. Since the sequence of $\eta$ is increasing and the set $A_{\SK_\eta}$ is strictly shrinking in the first component {of $\mathbb{Q}_S^3$}, by passing to the subsequence if necessary, the sequence
$\bsv_\eta$ consists of distinct points and each $p$-adic component of $\sz_\eta$ is nonzero.
For each $\eta$, take $\t_\eta={ -\sy_\eta/\sz_\eta}=(-y_{p,\eta}/z_{p,\eta})_{p\in S}\in \Q_S$. By passing to another subsequence if necessary, we may further assume that for each $p\in S$, $(|\t_{\eta}|_p)$ is an increasing sequence.
Moreover, the points cannot be contained in a compact set, or else the lattice would have accumulation points, so by taking a subsequence if necessary, we can choose $\bsv_\eta \in P(\Lambda) \cap A_{\SK_\eta}$ which is unbounded.


We have that the components of $\su_{\t_\eta}\bsv_\eta$ are either 
\[
\left(\begin{array}{c}
x_{p,\eta}\\
0\\
z_{p,\eta}
\end{array}\right)
\quad\text{or}\quad
\left(\begin{array}{c}
x_{p,\eta}- \frac{y_{p,\eta}^2}{2z_{p,\eta}}\\
0\\
z_{p,\eta}\end{array}\right).
\]

Notice that 
\[
\log\alpha_1(\su_{\t_\eta}\Lambda)
\ge \log\left(\prod_{p\in S}\|\su_{\t_\eta}\bsv_{\eta}\|^{-1}_p\right)
= \sum_{p\in S} \log\|\su_{\t_\eta}\bsv_{\eta}\|^{-1}_p
\]
so that one can obtain \cref{lower bound of AM} if we show lower bounds on the logarithm in each place.  
	
	In the $\infty$ place, 
 \[\begin{split}
	\log\norm{\su_{\t_\eta} \bsv_{\eta}}_\infty^{-1}
	&\ge\log\frac{1}{|\sx_{\eta}|_\infty + |\sz_{\eta}|_{\infty}}
	\quad\text{or}\quad
	\log\frac{1}{\left|\sx_{\eta}-\frac {\sy_{\eta}^2} {2\sz_{\eta}}\right|_\infty + |\sz_{\eta}|_{\infty}}\;\text{(respectively)}\\ 
	&\geq \log \frac{1}{2k_{\eta, \infty}^{-1/3+\eps}}\\
	&\ge  \frac{1/ 3 - \eps}{1+\eps}\log|\t_{\eta}|_\infty - \log2.\\
	\end{split}\]
	Similarly, for each $p\in S_f$, we have that
    \[\begin{split}
	\log\norm{\su_{\t_\eta}\bsv_{\eta}}^{-1}_p
	&\ge \log \frac 1 {\max\{|\sx_{\eta}|_p, |\sz_{\eta}|_p\}}
	\quad\text{or}\quad
	\log \frac 1 {\max\left\{\left|\sx_{\eta}-\frac {\sy_{\eta}^2}{2\sz_{\eta}}\right|_p, |\sz_{\eta}|_p\right\}}\;\text{(respectively)}\\
   &=\frac 1 3 \log p^{k_{p,\eta}}
   =\frac 1 3 \log |\t_{\eta}|_p.
	\end{split}\] 

	Thus we have
\begin{align*}\log(\alpha_1(\su_{\t_\eta}\Lambda))
&\geq\left[{ \frac{1/ 3 - \eps}{1+\eps}}
 \log|\t_{\eta}|_{\infty} -\log2+  \frac{1}{3} \sum_{p\in S_f}\log(|\t_\eta|_p) \right]\\
& \geq \left[-\log2+ { \frac{1/ 3 - \eps}{1+\eps}}
\sum_{p\in S}\log(|\t_\eta|_p) \right].\\
\end{align*}
Dividing by $\log \prod_{p\in S} |\t_\eta|_p$, we have
\begin{align*}
\frac{\log(\alpha_1(\su_{\t_\eta}\Lambda))}{ \log \prod_{p\in S} |\t_\eta|_p}
&\geq  \left[-\frac{\log2}{\log \prod_{p\in S} |\t_\eta|_p} +{  \frac{1/ 3 - \eps}{1+\eps}}
\right] \geq \frac{1}{3} -\delta,
\end{align*}
where the last inequality holds when
\[
{  \frac{1}{3} \left(1-\frac{1-3\epsilon}{1+\epsilon}\right) \leq \frac{\delta}{2}  }
\quad\text{and}\quad
\frac{\log2}{\log \prod_{p\in S} |\t_\eta|_p }\leq \frac{\delta}{2}.
\]
Since the product of the norms of our $\t_\eta$ diverges to infinity as $\eta$ goes to infinity, one can take $\eta_0>0$ so that the above is true for all $\eta > \eta_0$.
We now have the set of full measure for each $\delta$, and taking the intersection of these sets yields the full measure set where we have the desired lower bound.

\noindent\textbf{The general case $\dim \ge 3$.}\quad
In this case, each $p$-adic component of $\su_\t=(U_{t_p})_{p\in S}$ consists of Jordan blocks of the form provided in \cref{Jordan block}. Note that the number of blocks and their sizes would be different in each place. However, as we can see in the proof for the 3-dimensional case, it is irrelavant to our argument, since we will define the set $A_{\SK}$ for $\SK\in \N\times \prod_{p\in S_f}p^{d\N}$ as the product of the set $A_{\SK}^{(p)}$, where $A_{\SK}^{(p)}$ is determined by inequalities only relavant to the Jordan normal form $U_{t_p}$ over $\Q_p$.

For each $p\in S$, without loss of generality, we may assume the bottom block of $U_{t_p}$ has size {$\ell\geq 2$}. Thus there are polynomials $\widetilde{f}_j(\bx)$ with coefficients in ${ {t_p}^\ell/\ell!}$ for $j=1,\ldots, d-2$ and {$\ell\geq 0$} so that
$$
U_{t_p} \bx = \begin{pmatrix} \widetilde{f}_1(\bx) \\ \vdots \\ \widetilde{f}_{d-2}(\bx) \\ x_{d-1} + t_p x_d \\ x_d\end{pmatrix},
$$
where $\bx={\tp{(x_1, \ldots, x_d)}}\in \Q_p^d$.
Setting $t_p =-x_{d-1}/x_d$, we obtain rational functions $f_j(\bx)$ for $j=1,\ldots, d-2$, in the variables excluding $x_j$, and with only $x_d$ in the denominator so that
	$$U_{t_p} \bx = \begin{pmatrix} x_1 -{f}_1(\bx) \\ \vdots \\ x_{d-2} - {f}_{d-2}(\bx) \\ 0 \\ x_d\end{pmatrix}.$$
	
Now we can set $A_\SK^{(p)}$ so that in the infinite place,
\begin{enumerate}
	\item $0< |x_j - f_j(\bx)| \leq k_\infty^{-1/d}$ for $j\leq d-2$;
	\item {$|x_{d-1}| \leq k_\infty^{(d-1)/d},\, \frac{x_{d-1}}{x_d} <0$};
	\item ${|x_d|}\in [k_\infty^{-1/d-\eps} ,k_\infty^{-1/d+\eps}]$.
\end{enumerate}
And for each $p\in S_f$, we have
\begin{enumerate}
	\item $x_j - f_j(\bx) \in  p^{k_p/d}\Z_p - p^{(k_p/d)+1}\Z_p$ for $j\leq d-2$;
	\item ${  x_{d-1}} \in p^{-(d-1)k_p/d}\Z_p - p^{-((d-1)k_p/d) +1}\Z_p  $;
	\item $x_{d} \in p^{k_p/d}\Z_p - p^{(k_p/d)+1}\Z_p$.
\end{enumerate}

In this case the volume is given by
\[
\vol_S(A_\SK) = 2^{(d-1)}(k_\infty^\eps - k_\infty^{-\eps}) \prod_{p\in S_f} \left(1-\frac 1 p\right)^{d}
\]
so that $\vol_S(A_{\SK})\rightarrow \infty$ as $k_\infty\rightarrow \infty$, hence when $\SK\rightarrow \infty$. The proof now proceeds exactly as in case 1 with $3$ replaced by $d \geq 3$ and $\log(d-1)$ in place of $\log2$.

\noindent\textbf{The case $\dim=2$}.\quad
We first remark that as in \cref{cor:limit rand minkowski} for the case of $\dim\ge 3$, we didn't demand the volume $\vol_p(A_\SK^{(p)})$ for $p\in S_f$ diverging to infinity, as $k_p$ goes to infinity.
However, to use \cref{cor:limit rand minkowski} for the 2-dimensional case, we need $\vol_p(A_{\SK}^{(p)})$ diverging for all $p\in S$.

For any positive $\delta< 1/2$, choose {$\epsilon<1$ so that $\frac{1}{\epsilon}\in \N$ and  $\frac{1}{2}\left(1-\frac{1-2\epsilon}{1+\epsilon}\right) <\delta$.} The canonical unipotent one-parameter subgroup $\su_\t$ is
\[
\left(\begin{array}{cc}
1 & \t\\
0 & 1\end{array}\right),\;\t\in \Q_S.
\]

For each $\SK=(k_\infty, p_1^{k_{p_1}}, \ldots, p_s^{k_{p_s}})\in \N\times \prod_{p\in S_f} p^{(2/\eps)\N}$, define $A_\SK=\prod_{p\in S} A_{\SK}^{(p)}$ as the set of ${\tp{(y_p,z_p)}}\in \Q_p^2$ for which
\[\begin{cases}
|\sy_\infty|_\infty \le k_\infty^{1/2}, {\frac{y_\infty}{z_\infty} <0}\\
k_\infty^{-1/2-\eps} < |z_\infty| < k_\infty^{-1/2+\eps};
\end{cases}
\quad\text{or if }p\in S_f,\quad
\begin{cases}
{ y_p}\in p^{-k_p/2}\Z_p - p^{-(k_p/2)+1}\Z_p,\\
z_p\in p^{k_p(1/2-\eps)}\Z_p - p^{k_p(1/2 + \eps)}\Z_p;
\end{cases}
\]
so that each of volumes
\[
\vol_\infty\left(A_\SK^{(\infty)}\right)=2(k_\infty^\eps - k_\infty^{-\eps})
\quad\text{and}\quad
\vol_p\left(A_\SK^{(p)}\right)=\left(1-\frac 1 p\right)(p^{\eps k_p} - p^{-\eps k_p})
\]
diverges when $k_\infty,\;k_p \rightarrow \infty$, respectively so that one can proceed the argument used for general dimensional cases.
In particular, one can obtain the sequence $(\t_\eta)_{\eta}$ such that for each $p\in S$,
\[
\log\|\su_{\t_\eta}\bv_{\eta}\|_p^{-1} 
\ge \left(\frac 1 2 -\eps\right) \frac 1 {1+\eps} \log|\t_{\eta}|_p
\ge \left(\frac 1 2 - \delta\right) \log|\t_{\eta}|_p,
\]

where the last inequality follows from the choice of $\eps$ in the beginning.
\end{proof}

\bibliographystyle{alpha}
\bibliography{Sources}
\end{document}